\documentclass[10pt]{article}
\usepackage{fullpage,graphicx, amsmath,amsthm,epstopdf,tikz,enumitem,color, lineno, float,comment, amsfonts, wasysym,amssymb}
\setlist{nolistsep}
\usepackage[mathscr]{euscript}
\usepackage[colorlinks=true,allcolors=blue]{hyperref}
\usepackage[capitalise, noabbrev, nameinlink]{cleveref}

\DeclareMathOperator{\pw}{pw}
\usepackage{rotating}
\definecolor{purple}{RGB}{138,43,226}
\usepackage{thmtools}
\usepackage{thm-restate}

\definecolor{azure(colorwheel)}{rgb}{0.0, 0.5, 1.0}
\definecolor{Green}{RGB}{34, 139, 34}
\definecolor{USABlue}{RGB}{0,32,91}
\definecolor{USARed}{RGB}{191,13,62}
\definecolor{USALightBlue}{RGB}{0,169,224}

\usepackage{etoolbox}
\usepackage{cleveref}

\usetikzlibrary{backgrounds}

\usepackage{caption}
\usepackage[subrefformat=simple,labelformat=simple]{subcaption}

\newlist{statement}{enumerate}{1}
\setlist[statement]{label=\textup{(\alph*)},ref=\doublelabel{\textup{(\alph*)}},before=\setcrefdoublealias,nosep}
\crefname{statementi}{Statement}{Statements} 

\makeatletter
\newcommand*\doublelabel[1]{\protect\@twolabels{#1}{\@currentlabel#1}}
\let\@twolabels\@firstoftwo

\def\setcrefdoublealias{%
	\begingroup\edef\x{\endgroup%
		\noexpand\crefalias{statementi}{%
			\noexpand\protect\noexpand\@twolabels%
			{statementi}{\expandafter\@extractcounterfromcreflabel\cref@currentlabel\end@extractcounterfromcreflabel}%
		}%
	}\x%
}
\def\@extractcounterfromcreflabel[#1]#2\end@extractcounterfromcreflabel{#1}

\newcommand*\versionWithTheorem[1]{\@ifstar{\versionWithTheorem@aux{#1*}}{\versionWithTheorem@aux{#1}}}
\newcommand*\versionWithTheorem@aux[2]{\begingroup\let\@twolabels\@secondoftwo#1{#2}\endgroup}
\makeatother

\DeclareRobustCommand*\crefWithTheorem{\versionWithTheorem\cref}

\DeclareFontFamily{U}{FdSymbolA}{}
\DeclareFontShape{U}{FdSymbolA}{m}{n}{
	<-> s * [1.5] FdSymbolA-Book
}{}
\DeclareFontShape{U}{FdSymbolA}{m}{b}{
	<-> s * [1] FdSymbolA-Medium
}{}
\DeclareSymbolFont{fdsymbols}{U}{FdSymbolA}{m}{n}
\SetSymbolFont{fdsymbols}{bold}{U}{FdSymbolA}{m}{b}
\DeclareMathSymbol{\upY}{\mathbin}{fdsymbols}{45}
\DeclareMathSymbol{\downY}{\mathbin}{fdsymbols}{47}
\DeclareMathSymbol{\hourglass}{\mathbin}{fdsymbols}{43}

\mathchardef\mhyphen="2D

\newtheorem{theorem}{Theorem}[section]

\newtheorem{lemma}[theorem]{Lemma}
\newtheorem{corollary}[theorem]{Corollary}
\newtheorem{proposition}[theorem]{Proposition}
\crefname{proposition}{Proposition}{Propositions}

\numberwithin{figure}{section}
\newtheorem{observation}[theorem]{Observation}
\newtheorem{question}[theorem]{Question}
\crefname{question}{Question}{Questions}
\crefname{observation}{Observation}{Observations}
\newtheorem{conjecture}[theorem]{Conjecture}
\crefname{conjecture}{Conjecture}{Conjectures}

\theoremstyle{remark}
\newtheorem{remark}[theorem]{Remark}
\newtheorem{claim}{Claim}
\crefname{claim}{Claim}{Claims}

\newcommand{\thistheoremname}{}
\newtheorem{genericthm}[theorem]{\thistheoremname}

\newtheorem*{genericthm*}{\thistheoremname}
\newenvironment{namedthm*}[1]
{\renewcommand{\thistheoremname}{#1}%
	\begin{genericthm*}}
	{\end{genericthm*}}

\newcommand{\Ll}{\mathcal{L}}
\newcommand{\B}{\mathcal{B}}
\def\N{\mathcal{N}}
\def\A{\mathcal{A}}

\def\TS{\mathcal{T}}
\def\mZ{\mathbb{Z}}
\def\G{\mathcal{G}}

\definecolor{Green}{RGB}{34, 139, 34}
\def\Binf{\B_\infty}
\def\BinfB{\Binf^{\fam0 B}}
\def\BinfU{\Binf^{\fam0 U}}
\def\Bs{\B_{\fam0 S}}
\def\Bn{\B_{\fam0 N}}
\def\TSinf{\TS_\infty}
\def\TSinfB{\TSinf^{\fam0 B}}
\def\TSinfone{\TSinf^1}

\def\turan#1{Tur\'{a}n}
\def\erdos#1{Erd\H{o}s}
\def\sos#1{S\'{o}s}
\def\wozniak#1{Wo\'{z}niak}
\def\patkos#1{Patk\'{o}s}
\def\ex{{\fam0 ex}}
\def\exc{\ex_{\fam0 c}}

\title{Forbidding the subdivided claw as a subgraph or a minor}
\author{Sarah Allred\thanks{Department of Mathematics and Statistics, University of South Alabama, Mobile, AL, USA (sarahallred@southalabama.edu).}
\and M.~N.~Ellingham\thanks{Department of Mathematics,  Vanderbilt University, Nashville, TN, USA (mark.ellingham@vanderbilt.edu).  Supported by Simons Foundation award MPS-TSM-00002760.}
}

\date{} 
\begin{document}
\maketitle
\begin{abstract}
Let $Y$ be the \emph{subdivided claw}, the $7$-vertex tree obtained from a claw $K_{1,3}$ by subdividing each edge exactly once.  We characterize the graphs (finite and infinite) that do not have $Y$ as a subgraph, or, equivalently, do not have $Y$ as a minor.  This work was motivated by a problem involving VCD minors.  A graph $H$ is a \emph{vertex contraction-deletion minor}, or \emph{VCD minor}, of a graph $G$ if $H$ can be obtained from $G$ by a sequence of vertex deletions or contractions of all edges incident with a single vertex.  Our result is a key step in describing $K_{1,3}$-VCD-minor-free line graphs.  We also characterize graphs that forbid each subtree of $Y$.
We discuss the relevance of our results for \turan. numbers of trees, and pathwidth and growth constants for graphs without a particular tree as a minor.
\end{abstract}

\section{Introduction}\label{Sec:Intro}
There are several results on the structure of graphs that exclude a tree or forest as a minor.  However, there is little data on what graphs excluding specific small trees look like.  Information of this kind can help to generate new general questions.  In this paper we consider the \emph{subdivided claw}, the $7$-vertex graph obtained by subdividing each edge of $K_{1,3}$, which we denote by $Y$.
We consider the $T$-minor-free graphs where $T$ is either $Y$ or one of its subtrees.  These are equivalent to the graphs without $T$ as a subgraph, and we characterize those graphs in both finite and infinite situations.  Our characterizations are relevant to some general questions that we discuss at the end of the paper, involving \turan. numbers for trees, and pathwidth and growth constants for graphs with a tree as a forbidden minor.

For this paper, all graphs will be simple.  Any notation not defined here follows \cite{west}.
The size of a largest matching in a graph is denoted $\alpha'(G)$.
We will consider several well-known graph relations whose definitions follow.
A graph $H$ is a \emph{subgraph} of a graph $G$ if $H$ can be obtained from $G$ by a sequence of vertex and edge deletions. 
A graph $H$ is a \emph{minor} of a graph $G$ if $H$ can be obtained from $G$ by a sequence of vertex deletions, edges deletions, and edge contractions.
A graph $H$ is a \emph{topological minor} of a graph $G$ if $H$ is isomorphic to a subdivision of a subgraph of $G$.

It is often desirable to relate a particular class of graphs to a set of forbidden substructures.  There are a variety of such results using different graph relations. 
Bipartite graphs are characterized by forbidding odd cycles as subgraphs.
Kuratowski \cite{kuratowski} characterized planar graphs using the forbidden topological minors $K_5$ and $K_{3,3}$.  Wagner \cite{wagner} proved that the same two graphs were forbidden as minors.
In \cite{forlinegraphs},  Beineke proved that line graphs can be characterized by nine forbidden induced subgraphs.
Chordal graphs are characterized by having no induced cycles of length 4 or more.
In \cite{graphminorsXX}, Robertson and Seymour proved that every minor-closed class has a finite list of forbidden minors.

Since $Y$ is a tree that has only one vertex of degree at least 3, having $Y$ as a subgraph is equivalent to having $Y$ as a topological minor.
Further, since each vertex of $Y$ has degree at most 3, having $Y$ as a topological minor is equivalent to having $Y$ as a minor.  
Similarly, for any subtree $T$ of $Y$, having $T$ as a subgraph is equivalent to having $T$ as a topological minor or minor.

Our initial motivation for studying graphs without $Y$ as a subgraph or minor arose from a question involving a graph relation called a \emph{vertex-contraction-deletion minor}.
A \emph{vertex contraction} of a vertex $v$ in a graph $G$ is the operation of identifying $v$ and all neighbors of $v$ into a single vertex.   
This is equivalent to contracting all edges incident with $v$.
A graph $H$ is a \emph{vertex-contraction-deletion minor} (VCD minor) of a graph $G$ if $H$ can be obtained from $G$ by a sequence of vertex deletions and vertex contractions.
The VCD minor relation is a slightly weaker relation than the \emph{t-minor} relation \cite{BS10}.
We are interested in characterizing $K_{1,3}$-VCD-minor free graphs, and an important subclass of this family is the $K_{1,3}$-VCD-minor free line graphs.
Forbidding a $K_{1,3}$ as a VCD minor in a line graph $G=L(H)$ where $\alpha'(H)\ge 4$ is equivalent to forbidding $Y$ as a subgraph of $H$.  The results in this paper will be used in a forthcoming paper on this topic.
 
We note that forbidding $Y$ as a subgraph has already been used to characterize the class of trees known as \emph{caterpillars}, in which there is a path (the \emph{spine}) such that every edge is incident with a vertex of that path.  This is a special case of our main result.

\begin{theorem}[West {\cite[Theorem 2.2.19]{west}}; see also {\cite[p.~172]{kl-fmpw2}}]\label{thm:treenoycat}
A tree does not have $Y$ as a subgraph if and only if it is a caterpillar.
\end{theorem}

We will now define some structures necessary to describe $Y$-subgraph free graphs.
A \emph{bead} is a graph $G$ with a set $V_1$ of one or two special vertices designated as \emph{primary} vertices; the other vertices are \emph{secondary} vertices.
When necessary we describe a bead as an ordered pair $(G,V_1)$, but usually we just name the graph $G$.
The graph of a bead is one of the following: $K_4$, $K_{2,1,1}$, $K_{1,1,t_1}$ with $t_1 \ge 0$, or $K_{2,t_2}$ with $t_2 \ge 2$.
The primary vertices in $K_{1,1,t_1}$ are the parts of size one.
The primary vertices in $K_{2,t_2}$ and $K_{2,1,1}$ are the vertices in the part of size two.
One vertex of $K_4$ is designated as a primary vertex.
Note that for us $K_{2,1,1}$ is not the same as $K_{1,1,2}$ because each has a different set of primary vertices.  The beads are illustrated in \cref{fig:beads}, where the white (open) vertices are the primary vertices.

Beads can be ``strung" together at primary vertices. 
This means we can identify a primary vertex in one bead with a primary vertex in a different bead.
Each primary vertex can only be identified with one other primary vertex.

\begin{figure}[h]
 \begin{subfigure}[b]{.23\textwidth}
\begin{center}
     \begin{tikzpicture}
        [scale=1,auto=left,every node/.style={circle, fill, inner sep=0 pt, minimum size=2mm, outer sep=0pt},line width=.4mm]
            \node (1) at (0:0cm){};
            \node (2) at (120:1.25cm){};
            \node (3) at (240:1.25cm){};
            \node[fill=white,draw] (4) at (0:1.25cm){};
            \begin{scope}[on background layer]
                \foreach \i in {120,240, 0}{
                \draw[line width=.4mm] (0:0cm) to (\i:1.25cm); 
                \draw[line width=.4mm] (\i:1.25cm) to (\i+120:1.25cm);
                }
            \end{scope}
	\end{tikzpicture}
\end{center}
        \caption{$K_{4}$}
    \end{subfigure}
    \begin{subfigure}[b]{.23\textwidth}
\begin{center}
     \begin{tikzpicture}
        [scale=1,auto=left,every node/.style={circle, fill, inner sep=0 pt, minimum size=2mm, outer sep=0pt},line width=.4mm]
            \node[fill=white,draw](1) at (0,1) {};
            \node (2) at (1,0){};
            \node (3) at (1,2){};
            \node[fill=white,draw] (4) at (2,1){};
            \begin{scope}[on background layer]
                \draw[line width=.4mm] (1.center) to (3.center) to (2.center) to (4.center) to (3.center); 
                \draw[line width=.4mm] (1.center) to (2.center);
            \end{scope}
	\end{tikzpicture}
\end{center}
        \caption{$K_{2,1,1}$}
    \end{subfigure}
    \begin{subfigure}[b]{.23\textwidth}
        \begin{center}
     \begin{tikzpicture}
        [scale=1,auto=left,every node/.style={circle, fill, inner sep=0 pt, minimum size=2mm, outer sep=0pt},line width=.4mm]
            \node[fill=white,draw] (1) at (0,1.1){};
            \node[fill=white,draw] (2) at (2,1.1) {};
            \node (3) at (1,0){};
            \node (4) at (1,.75){};
            \node (5) at (1,1.5){};
            \node (6) at (1,2.25){};
            \begin{scope}[on background layer]
            \draw[line width=.4mm] (1.center) to (2.center);
                \foreach \i in {3,4,5,6}{
                \draw[line width=.4mm] (1.center) to (\i.center);
                \draw[line width=.4mm] (2.center) to (\i.center);
                }
            \end{scope}
	\end{tikzpicture}
\end{center}
        \caption{$K_{1,1,t_1}$ with $t_1=4$}
    \end{subfigure}
     \begin{subfigure}[b]{.23\textwidth}
\begin{center}
     \begin{tikzpicture}
        [scale=1,auto=left,every node/.style={circle, fill, inner sep=0 pt, minimum size=2mm, outer sep=0pt},line width=.4mm]
            \node[fill=white,draw] (1) at (0,1.1){};
            \node[fill=white,draw] (2) at (2,1.1) {};
            \node (3) at (1,0){};
            \node (4) at (1,.75){};
            \node (5) at (1,1.5){};
            \node (6) at (1,2.25){};
            \begin{scope}[on background layer]
                \foreach \i in {3,4,5,6}{
                \draw[line width=.4mm] (1.center) to (\i.center);
                \draw[line width=.4mm] (2.center) to (\i.center);
                }
            \end{scope}
 	\end{tikzpicture}
\end{center}
        \caption{$K_{2,t_2}$ with $t_2=4$}
    \end{subfigure}
    \caption{Beads}
       \label{fig:beads}
\end{figure}

Consider a connected graph $G$ created by stringing beads together.  We form a bipartite auxiliary graph $A$ by taking one black vertex for each bead and one white vertex for each primary vertex, where a bead vertex is joined to its primary vertices.
Then $G$ is a \emph{strand} if $A$ is a path, and a \emph{necklace} if $A$ is a cycle.
For a strand, we could have $K_4$ with one primary vertex as the first and/or last bead; necklaces do not have $K_4$ as a bead.
A necklace with exactly two beads cannot have both beads of the form $K_{1,1,t_1}$ because stringing two beads of this kind would create parallel edges, and we only allow simple graphs.
A vertex of degree $1$ is a \emph{leaf} or \emph{pendant vertex}, and an edge incident with a leaf is a \emph{pendant edge}.
A \emph{spiked necklace (strand)} is a necklace (strand) that is allowed to have additional pendant edges called \emph{spikes} incident with primary vertices that are in exactly two beads.
\cref{fig:structuresB} illustrates a spiked strand and a spiked necklace; again, the white (open) vertices are the primary vertices.

We now explain some of the restrictions imposed by our definitions.  
We could allow $K_{2,1}$ as a bead, but this can be considered as two $K_{1,1,0}$ beads strung together.
We could allow spikes incident with primary vertices that are in only one bead, which must be a bead at the end of a strand, but then one of those spikes can be considered as a new $K_{1,1,0}$ bead.
Thus, these restrictions eliminate ambiguity about the number of beads in a spiked strand.
If a necklace has exactly two beads, they cannot both have the form $K_{1,1,t}$ because that would create multiple edges, and our graphs are simple.
There are still some ambiguities about whether certain graphs should be regarded as spiked strands or spiked necklaces, which we do not try to eliminate.  These occur for spiked necklaces with at most four beads.  

\begin{figure}[h!]
\begin{subfigure}[c]{.55\textwidth}
\centering
     \begin{tikzpicture}
        [scale=1,auto=left,every node/.style={circle, fill, inner sep=0 pt, minimum size=2mm, outer sep=0pt},line width=.4mm]
            \node (a) at (0,1){};
            \node (b) at (-.5,1.87){};
            \node (c) at (-.5,.13){};
            \node[fill=white,draw] (1) at (1,1){};
            \node (2) at (2,0){};
            \node (3) at (2,2){};
            \node[fill=white,draw] (4) at (3,1){};
            \node[fill=white,draw] (1a) at (3,1){};
            \node[fill=white,draw] (2a) at (5,1) {};
            \node (3a) at (4,0){};
            \node (4a) at (4,.65){};
            \node (5a) at (4,1.35){};
            \node (6a) at (4,2){};
            \node (1p1) at (1,0){};
            \node (1p2) at (.75, 2){};
            \node (2p1) at (3,0){};
             \node (2b) at (6,0){};
            \node (3b) at (6,2){};
            \node[fill=white] (space2) at (4,3.25){};
            \node[fill=white] (space) at (4,-1){};
            \node[fill=white,draw] (4b) at (7,1){};
            \node[fill=white, draw] (5d) at (8,1) {};
            \node (5b) at (7,0) {};
            \node (5c) at (7,2) {};
            \begin{scope}[on background layer]
            \draw[line width=.4mm] (4) to (2p1);
            \draw[line width=.4mm] (1p2) to (1) to (1p1);
            \draw[line width=.4mm] (3b) to (2a) to (2b) to (3b) to (4b) to (2b);
            \draw[line width=.4mm] (b) to (c) to (1) to (b);
                \foreach \i in {c,b,1}{
                \draw[line width=.4mm] (a) to (\i); 
                }
                 \draw[line width=.4mm] (1.center) to (3.center) to (2.center) to (4.center) to (3.center); 
                \draw[line width=.4mm] (1.center) to (2.center);
                \draw[line width=.4mm] (1a.center) to (2a.center);
                \foreach \i in {3a,4a,5a,6a}{
                \draw[line width=.4mm] (1a.center) to (\i.center);
                \draw[line width=.4mm] (2a.center) to (\i.center);}
                \draw[line width=.4mm] (5c) to (4b) to (5b); 
                \draw[line width=.4mm] (5d) to (4b);
                \end{scope}
	\end{tikzpicture}

\caption{A spiked strand}
\label{fig:spikestrand}
\end{subfigure}
\begin{subfigure}[c]{.4\textwidth}
\centering
     \begin{tikzpicture}
        [scale=1,auto=left,every node/.style={circle, fill, inner sep=0 pt, minimum size=2mm, outer sep=0pt},line width=.4mm]
\node[fill=white,draw] (1) at (1,1){};
            \node[fill=white,draw] (2) at (3,1) {};
            \node (3) at (2,0){};
            \node (4) at (2,.5){};
            \node[fill=white,draw] (1a) at (1,3) {};
            \node (2a) at (.25,2) {};
            \node (3a) at (1.75,2){};
            \node[fill=white, draw] (1b) at (3,3){};
            \node (2b) at (2,2.75){};
            \node (3b) at (2,3.5){};
            \node (4b) at (2,4.25){};
            \node (1c) at (3.75,2){};
            \node (2c) at (4.5,2){};
            \node (3c) at (5.25,2){};
            \node (p1) at (4.25,3.2) {};
            \node (p2) at (3.5,3.75){};
            \node (p3) at (.75,4){};
            \node (p4) at (.25, .25){};
            \begin{scope}[on background layer]
            \draw[line width=.4mm] (1.center) to (2.center);
                \foreach \i in {3,4}{
                \draw[line width=.4mm] (1.center) to (\i.center);
                \draw[line width=.4mm] (2.center) to (\i.center);
                }
                \draw[line width=.4mm] (3a) to (2a) to (1a) to (3a) to (1) to (2a);
                \foreach \j in {2b,3b,4b}{
                \draw[line width=.4mm] (1a) to (\j.center);
                \draw[line width=.4mm] (1b) to (\j.center);
                }
                \foreach \k in {1c,2c,3c}{
                \draw[line width=.4mm] (1b.center) to (\k.center);
                \draw[line width=.4mm] (2.center) to (\k.center);
                }
                \draw[line width=.4mm] (1b) to (2);
                \draw[line width=.4mm] (p1) to (1b) to (p2);
                \draw[line width=.4mm] (p3) to (1a);
                \draw[line width=.4mm] (p4) to (1);
            \end{scope}
	\end{tikzpicture}
 
\caption{A spiked necklace}
\label{fig:spikednecklace}
\end{subfigure}
\caption{Structures in $\B$}
\label{fig:structuresB}
\end{figure}

A graph is a member of the family \emph{$\B$} if and only if it is a spiked strand or spiked necklace.  All graphs in $\B$ are connected.

We say that two nonadjacent vertices $u$ and $v$ in a graph $G$ are \emph{clones} in $G$ if $N(u)=N(v)$.
This is an equivalence relation and the equivalence classes are called\emph{clone classes}.
A \emph{leaf class} is a clone class whose elements are leaves in $G$.
By \emph{cloning} a vertex $v$ in $G$, we mean adding new vertices $v_1,\dots , v_k$ to form $G'$ such that $v, v_1, \dots, v_k$ are pairwise nonadjacent and $N_{G'}(v_i)=N_G(v)$ for $1\le i \le k$.

We can now state the main theorem of this paper.

\begin{restatable}{theorem}{structthm}\label{thm:noY}
A connected graph $G$ has no subgraph isomorphic to $Y$ (is $Y$-minor-free) if and only if either
\begin{statement}
\item\label{smallG} $G$ is obtained from a graph with at most six vertices by optionally cloning leaves, or 
\item\label{infam} $G$ belongs to the family $\B$, i.e., $G$ is a spiked strand or spiked necklace. \end{statement}
\end{restatable}

In \cref{Sec:PathCase,Sec:CycleCase,Sec:MainResult,sec:Subtree}, we consider all graphs to be finite.
We divide the forwards direction of the proof of \cref{thm:noY} into two cases, depending on whether $G$ has a long cycle or not.  \cref{Sec:PathCycle} proves some basic structural results that apply in both cases,
\cref{Sec:PathCase} addresses the case where there is no long cycle (giving spiked strands), and \cref{Sec:CycleCase} addresses the case where there is a long cycle (giving spiked necklaces).
We complete the proof in \cref{Sec:MainResult}.
In \cref{sec:Subtree}, we characterize graphs without $T$ as a subgraph (or minor) for subtrees $T$ of $Y$.
Infinite analogues of the results in \cref{Sec:MainResult,sec:Subtree} appear in \cref{sec:BoundedPathLength}.
In \cref{sec:pathwidth}, we make some observations on the number of edges, pathwidth, and growth constant of graphs without a $Y$, which raise some larger issues.

\section{Basic structural results}\label{Sec:PathCycle}
In this section, we prove some basic facts about the structure of a connected graph that does not have $Y$ as a subgraph.  Throughout \cref{Sec:PathCycle,Sec:PathCase,Sec:CycleCase} $G$ denotes such a graph.  We derive the structure by starting with a longest path, which if possible has ends that are leaves of $G$.

The \emph{pointiness} of a path $P$ in $G$ is the number of endpoints of $P$ that are leaves of $G$ (namely $0$, $1$, or $2$).  A path with no edges has only one endpoint, hence pointiness $0$ or $1$.  We talk about one path being \emph{pointier} than another if the pointiness of the first path is higher, and a \emph{pointiest} path in a collection is one of maximum pointiness.

For us, $P_k$ means a path with $k$ vertices.  A \emph{star} is a graph $K_{1,t}$ for some $t \ge 0$, which includes $K_1 \cong K_{1,0}$ and $K_2 \cong K_{1,1}$.  A \emph{double star} is a graph obtained from $P_4$ by optionally cloning leaves.  A \emph{triangle} is $K_3 = C_3$.

For this section and the next, we fix a pointiest longest path $P = v_0 v_1 \dots v_\ell$ of $G$.  (The first few results below are valid for any longest path $P$; we do not use pointiness until \cref{lem:endchord}.)
By a \emph{chord} we mean a chord of $P$, i.e., $e=v_j v_k \in E(G)$ with $k \ge j+2$; we call $e$ an \emph{$i$-chord} if $k-j = i$.
The vertices $v_{j+1},\dots, v_{k-1}$ are the \emph{enclosed vertices} of $e$, and we say that $e$ \emph{covers} the edges of the path $v_j v_{j+1} \dots v_k$.
Two chords $v_i v_j$ and $v_p v_q$, with $i < j$ and $p < q$, are said to \emph{cross} if either $i < p < j < q$ or $p < i < q < j$.
By a \emph{vee} we mean a path $v_i w v_{i+2}$ where $w \notin V(P)$.
The vertex $v_{i+1}$ is the \emph{enclosed vertex} of the vee, and we say that the vee \emph{covers} the edges of the path $v_i v_{i+1} v_{i+2}$.
A vee $v_i w v_{i+2}$ is said to \emph{cross} another vee or a chord if $v_{i+1}$ is an endpoint of the other vee or chord.

A subgraph $H$ of $G$ is \emph{edge-dominating} if every edge of $G$ is incident with a vertex of $H$ (i.e., $V(H)$ is a vertex cover).  Let $L_i=N_G(v_i)-V(P)$ for $0 \le i \le \ell$, and $\Ll = \bigcup_{i=0}^\ell L_i$.  The following is a key part of our argument.

\begin{lemma}\label{lem:domset}
The path $P$ is edge-dominating.  Thus, $V(G)-V(P) = \Ll$, and $\Ll$ is an independent set in $G$.
\end{lemma}

\begin{proof}
Suppose there is an edge $e = wx$ with $w, x \notin V(P)$.  Then there is a path $Q$ with first edge $e$, at least two edges, no internal vertices in $P$, and ending at $v_i$ on $P$.  We cannot have $i \le 1$, otherwise $Q \cup v_i v_{i+1} \dots v_\ell$ would be a longer path than $P$.  Similarly, we cannot have $i \ge \ell-1$.  Therefore, $2 \le i \le \ell-2$, and $Q \cup P$ contains a $Y$ subgraph, which is a contradiction, so no such $e$ exists.  The claims about $\Ll$ follow easily.
\end{proof}

When $G$ is a tree, \cref{lem:domset} implies \cref{thm:treenoycat}.

Our analysis focuses on the relationships between the sets $L_i$, including whether vees are created, and on the existence of particular chords of the path $P$.  We first prove some lemmas based on $P$ being a longest path.

\begin{lemma}\label{lem:k0klempty}
$L_0$ and $L_{\ell}$ are empty.
\end{lemma}
\begin{proof}
    If either $L_0$ or $L_{\ell}$ is nonempty, then $P$ would not be the longest path in $G$.  The conclusion follows.
\end{proof}

Next we consider intersections between the sets $L_i$.

\begin{lemma}\label{lem:intersectKiKi1}
For $i\in\{0,1,\dots,\ell-1\}$, $L_{i}\cap L_{i+1} = \emptyset$.
\end{lemma}
\begin{proof}
If $w \in L_i\cap L_{i+1}$, then replacing $v_i v_{i+1}$ by $v_i w v_{i+1}$ gives a path longer than $P$.  Thus, $L_i\cap L_{i+1}=\emptyset$, as required.
\end{proof}

\begin{lemma}\label{lem:nonemptykikj}
If $L_i\cap L_j\ne \emptyset$ where $i<j$, then $j=i+2$, or $i=1$ and $j=\ell-1$.
\end{lemma}
\begin{proof}
By \cref{lem:k0klempty,lem:intersectKiKi1} we have $i \ge 1$, $j-i \ge 2$, and $j \le \ell-1$.  Therefore, $\ell \ge 4$.  If $\ell = 4$ we must have $i=1$ and $j=3$, which satisfies the conclusion.  So we may assume that $\ell \ge 5$.

Let $w\in L_i\cap L_j$ and suppose that $j > i+2$.
If $i\ge 2$, then the paths $v_i v_{i-1} v_{i-2}$, $v_i v_{i+1} v_{i+2}$ and $v_i w v_j$ form a $Y$ subgraph.  If $j \le \ell-2$, there is a symmetric $Y$ subgraph.  The only remaining case is $i=1$ and $j=\ell-1$, so either this holds or $j=i+2$. 
\end{proof}

The following lemma means that an edge in $E(G)-E(P)$ belongs to at most one vee, and a vertex of $V(G)-V(P) = \Ll$ has degree at most $2$.

\begin{lemma}\label{lem:nonemptyintersect3}
If $i$, $j$, and $k$ are distinct, then $L_i\cap L_j\cap L_k=\emptyset$.
\end{lemma}
\begin{proof}
Suppose that $L_i\cap L_j\cap L_k\ne \emptyset$, where $i < j < k$. 
\cref{lem:k0klempty,lem:nonemptykikj} imply that $i\ge 1$, $j\ge i+2$, and $j+2 \le k \le \ell-1$.     
If $w\in L_i\cap L_j\cap L_k$, then the paths $w v_i v_{i+1}$, $w v_j v_{j+1}$, $w v_k v_{k+1}$ form a $Y$ subgraph.
Thus, $L_i\cap L_j\cap L_k=\emptyset$, as required.
\end{proof}

\begin{lemma}\label{lem:veecross}
A vee does not cross a $2$-chord or another vee.
\end{lemma}

\begin{proof}
Suppose the vee $v_{i} w v_{i+2}$ crosses a $2$-chord $v_{i+1} v_j$.  We assume that $j = i-1$; the case $j=i+3$ is symmetric.  We have a path $v_0 v_1 \dots v_{i-1} v_{i+1} v_i w v_{i+2} v_{i+3} \dots v_\ell$ which is longer than $P$, a contradiction.  We can similarly find a path longer than $P$ when two vees cross.
\end{proof}

While a vee cannot cross a $2$-chord or another vee, we will see that it is possible to have multiple vees that cover the same pair of edges, or vees that cover the same pair of edges as a $2$-chord.

Pointiness is used to prove the following result.

\begin{lemma}\label{lem:endchord}
If $v_0 v_i$ is a chord, then $L_{i-1} = \emptyset$, and $L_{i-2}$ contains no leaves of $G$.  Similarly, if $v_{\ell-i} v_\ell$ is a chord, then $L_{\ell-i+1} = \emptyset$, and $L_{\ell-i+2}$ contains no leaves of $G$.
\end{lemma}

\begin{proof}
The two statements are symmetric, so we prove only the first.  Note that $i \ge 2$.  If $w \in L_{i-1}$ then $w v_{i-1} v_{i-2} \dots v_0 v_i v_{i+1} \dots v_\ell$ is a longer path than $P$.  If $L_{i-2}$ contains a leaf $w$ of $G$ then $w v_{i-2} v_{i-3} \dots\allowbreak v_0 v_i v_{i+1} \dots v_\ell$ is a longest path that is pointier than $P$.
\end{proof}

We now deal with the situations where $\ell \le 4$, so that we can focus on the general case where $\ell \ge 5$.
If $\ell=0$ or $1$ then $G$ is either $K_1$ or $K_2 \cong K_{1,1,0}$, and hence satisfies both conclusions of \cref{thm:noY}.  The cases where $2 \le \ell \le 4$ are handled by \cref{lem:lengthl2,lem:lengthl3,lem:lengthl4}.

\begin{lemma}\label{lem:lengthl2}
If $\ell = 2$, then $G$ is a triangle or a star.  Hence, $G$ satisfies \crefWithTheorem{smallG}.
\end{lemma}

\begin{proof}
By \cref{lem:k0klempty}, $L_0=L_2=\emptyset$, so $V(G)-V(P) = L_1$.  If $v_0 v_2 \notin E(G)$, then $G$ is a star, and if $v_0 v_2 \in E(G)$, then $L_1 = \emptyset$ by \cref{lem:endchord}, so $G$ is a triangle.
\end{proof}

\begin{lemma}\label{lem:lengthl3}
If $\ell=3$, then $G$ is a double star, a graph of order at most $4$, or a graph obtained by attaching pendant edges to exactly one vertex of a triangle.  Hence, $G$ satisfies \crefWithTheorem{smallG}.
\end{lemma}

\begin{proof}
By \cref{lem:k0klempty}, $L_0 = L_3 = \emptyset$, and by \cref{lem:intersectKiKi1}, $L_1 \cap L_2 = \emptyset$.
Therefore, if $P$ has no chords then $P$ is a double star.  If either $v_0 v_3 \in E(G)$, or $v_0v_2, v_1 v_3 \in E(G)$, then $L_1 = L_2 = \emptyset$ by \cref{lem:endchord}, and $G$ has order at most $4$.  So we may assume that there is exactly one chord, either $v_0 v_2$ or $v_1 v_3$.  Since these cases are symmetric, we may suppose that the only chord is $v_0 v_2$.  Then $L_1 = \emptyset$ by \cref{lem:endchord}, and the vertices in $L_2$ are clones of $v_3$, so $G$ is a triangle with pendant edges at one vertex.
\end{proof}

\begin{lemma}\label{lem:lengthl4}
If $\ell=4$, then $G$ satisfies \crefWithTheorem{smallG}, or $G$ is a spiked strand and satisfies \crefWithTheorem{infam}.
\end{lemma}
\begin{proof}
\setcounter{claim}{0}
\cref{lem:k0klempty} implies that $L_0=L_4=\emptyset$, and \cref{lem:intersectKiKi1} implies that $L_1 \cap L_2 = L_2 \cap L_3 = \emptyset$.

\begin{claim}\label{clm:a1}
We may assume that $v_0 v_4 \notin E(G)$.  Otherwise, $v_0 v_4 \in E(G)$, and there is a cycle $C$ with $V(C) = V(P)$. If $L_1 \cup L_2 \cup L_3 \ne \emptyset$, there is a path longer than $P$, which is impossible.  So, $G$ has five vertices, and \crefWithTheorem{smallG} holds.
\end{claim}

\begin{claim}\label{clm:a2}
By \cref{lem:endchord}, none of the following happen: (a) $L_1 \ne \emptyset$ and $v_0 v_2 \in E(G)$,
(b) $L_1$ contains a leaf and $v_0 v_3 \in E(G)$,
(c) $L_2 \ne \emptyset$ and either $v_0 v_3 \in E(G)$ or $v_1 v_4 \in E(G)$, (d) $L_3 \ne \emptyset$ and $v_2 v_4 \in E(G)$,
(e) $L_3$ contains a leaf and $v_1 v_4 \in E(G)$.
\end{claim}

Suppose first that $L_1 \cap L_3 \ne \emptyset$.  If $L_2 \ne \emptyset$ then there is a path longer than $P$, which is impossible, so $L_2 = \emptyset$.  By \cref{clm:a1,clm:a2} $v_0 v_4, v_0 v_2, v_2 v_4 \notin E(G)$, so the only possible chords are $v_0 v_3$, $v_1 v_3$ and $v_1 v_4$.
If $v_1 v_3 \in E(G)$ there is a $K_{1,1,t}$ bead, $t \ge 2$, with primary vertices $v_1, v_3$ and secondary vertices including $v_2$, $L_1 \cap L_3$ and possibly $v_0$ (if $v_0 v_3 \in E(G)$) or $v_4$ (if $v_1 v_4 \in E(G)$); the remaining edges form at most two $K_{1,1,0}$ beads and spikes.  If $v_1 v_3 \notin E(G)$ we have a similar situation but with a $K_{2,t}$ bead, $t \ge 2$.  In either case $G$ is a spiked strand.

Now suppose that $L_1 \cap L_3 = \emptyset$.  So all vertices in $L_1 \cup L_2 \cup L_3$ are leaves.  If $L_1 \ne \emptyset$ then by \cref{clm:a1,clm:a2} there are no chords incident with $v_0$.  So whether $L_1 = \emptyset$ or not, all elements of $L_1$ are clones of $v_0$.  Similarly, all elements of $L_3$ are clones of $v_4$.  All elements of $L_2$ are clones.  Therefore, \crefWithTheorem{smallG} holds.
\end{proof}

The results for small $\ell$ can be summarized as follows.

\begin{proposition}\label{prop:small}
Suppose $G$ is a connected graph with no $Y$ subgraph and $P = v_0 v_1 \dots v_\ell$ is a longest path in $G$, where $\ell \le 4$.  Then $G$ either satisfies \crefWithTheorem{smallG}, or $G$ is a spiked strand and satisfies \crefWithTheorem{infam}.
\end{proposition}

We now prove some results that apply when $\ell \ge 5$.

We first consider neighbors of the vertex enclosed by a vee: they must lie on $P$.

\begin{lemma}\label{lem:vwithpend}
Suppose that $\ell \ge 5$ and $1\le i\le \ell-1$.  Then we do not have $L_i\ne \emptyset$ and $L_{i-1}\cap L_{i+1} \ne \emptyset$.
\end{lemma}
\begin{proof}
Suppose that $L_i\ne \emptyset$ and $L_{i-1}\cap L_{i+1}\ne \emptyset$.  
Then $2 \le i \le l-2$ because $L_0 = L_{\ell} = \emptyset$.
Let $w\in L_i$ and $u\in L_{i-1}\cap L_{i+1}$.  Then $w \ne u$ by \cref{lem:intersectKiKi1}.
If $i \ge 3$ then the paths $v_{i-1} v_{i-2} v_{i-3}$, $v_{i-1} v_i w$, $v_{i-1} u v_{i+1}$ form a $Y$ subgraph, which is a contradiction.  If $i \le \ell-3$ there is a symmetric $Y$ subgraph.  Since $\ell \ge 5$, these cases cover all values of $i$.
\end{proof}

Similarly, neighbors of the vertex enclosed by a $2$-chord must lie on $P$.

\begin{lemma}\label{lem:2chordpendant}
If $\ell \ge 5$ and $v_{i-1} v_{i+1} \in E(G)$ where $1\le i\le \ell-1$, then $L_i = \emptyset$.
\end{lemma}
\begin{proof}
Suppose that $v_{i-1} v_{i+1} \in E(G)$, and assume that $L_i \ne \emptyset$.
Then $i \ne 1, \ell-1$ by \cref{lem:endchord}.
If $3 \le i\le \ell- 2$, then $G$ has a $Y$ subgraph centered at $v_{i-1}$ because $\ell \ge 5$.
If $i = 2$, then $G$ has a $Y$ subgraph centered at $v_{i+1}=v_3$.
In all cases we reach a contradiction, so $L_i = \emptyset$.
\end{proof}

To conclude this section, we characterize the possible chords of $P$.
\begin{lemma}\label{lem:pos3pchords}
If $\ell \ge 5$, then every chord of $P$ is one of the following: $v_i v_{i+2}$ for $0 \le i \le \ell-2$, $v_0 v_3$, $v_{\ell-3} v_\ell$, $v_0 v_{\ell-1}$, $v_0 v_\ell$, $v_1 v_{\ell-1}$, and $v_1 v_\ell$.
\end{lemma}
\begin{proof}
Suppose we have a chord $v_i v_j$ with $j-i \ge 3$.
If $i \ge 1$ and $j \le \ell-2$ the paths $v_j v_i v_{i-1}$, $v_j v_{j-1} v_{j-2}$, $v_j v_{j+1} v_{j+2}$ form a $Y$ subgraph.  There is a symmetric $Y$ subgraph if $i \ge 2$ and $j \le \ell-1$.  These cover all cases except $i=0$, $j=\ell$, and $(i,j)=(1,\ell-1)$.
If $i=0$ and $4 \le j \le \ell-2$ the paths $v_j v_0 v_1$, $v_j v_{j-1} v_{j-2}$, $v_j v_{j+1} v_{j+2}$ form a $Y$ subgraph.  There is a symmetric $Y$ subgraph if $2 \le i \le \ell-4$ and $j=\ell$.  These cover all cases with $i=0$ or $j=\ell$ except $(i,j) = (0,3)$, $(0,\ell-1)$, $(0,\ell)$, $(1,\ell)$, and $(\ell-3,\ell)$.
Thus, we must have $j-i=2$ or one of the six listed cases.
\end{proof}

\section{No long cycle}\label{Sec:PathCase}

In this section, we consider the case where (roughly) $G$ does not have a long cycle.
Specifically, we consider a pointiest longest path $P$, and we assume throughout this section that $\ell \ge 5$, $P$ does not have a chord from $v_0$ or $v_1$ to $v_{\ell-1}$ or $v_{\ell}$, and $L_1 \cap L_{\ell-1} = \emptyset$.  Thus, the only chords allowed are $v_0 v_3$, $v_{\ell-3} v_\ell$, or $v_i v_{i+2}$ for some $i$, and if $L_i \cap L_j \ne \emptyset$ for $i < j$, then we must have $j = i+2$.
We will address the other cases for $\ell \ge 5$ in \cref{Sec:CycleCase}.

Except in some situations covered by \crefWithTheorem{smallG}, $P$ does not have three consecutive 2-chords.

\begin{lemma}\label{lem:3con2chords}
If $P$ has three 2-chords $v_i v_{i+2}$, $v_{i+1} v_{i+3}$, and $v_{i+2} v_{i+4}$, where $0 \le i \le \ell-4$, then $\ell=5$ and $G$ satisfies \crefWithTheorem{smallG}.
\end{lemma}
\begin{proof}
Assume that the three given $2$-chords exist.

Suppose first that $\ell \ge 6$.
If $i=0$, then the paths $v_4 v_5 v_6$, $v_4 v_3 v_1$, $v_4 v_2 v_0$ form a $Y$ subgraph.
If $i = \ell-4$ there is a symmetric $Y$ subgraph.
If $i\in\{1,\dots,\ell-5\}$, then the paths $v_{i+2} v_i v_{i-1}$, $v_{i+2} v_{i+1} v_{i+3}$, $v_{i+2} v_{i+4} v_{i+5}$ form a $Y$ subgraph.

Suppose now that $\ell=5$.  
Either $i=0$ or $i=\ell-4=1$.  These cases are symmetric.
Without loss of generality, suppose $i=0$.
Then $L_1 = L_2 = L_3 = \emptyset$ by \cref{lem:2chordpendant}.
If $L_4 = \emptyset$ then \crefWithTheorem{smallG} holds.
So suppose $L_4 \ne \emptyset$.  Then all elements of $L_4$ are leaves.  By the assumptions of this section $v_0v_5, v_1v_5 \notin E(G)$, and $v_2v_5, v_3v_5 \notin E(G)$ by \cref{lem:endchord}.  So $v_5$ is a leaf, all elements of $L_4$ are clones of $v_5$, and \crefWithTheorem{smallG} holds.
\end{proof}

Except in some situations covered by \crefWithTheorem{smallG}, a $3$-chord does not cross $2$-chords or the other possible $3$-chord.

\begin{lemma}\label{lem:3chordcross}
If the two $3$-chords $v_0 v_3$ and $v_{\ell-3} v_\ell$ cross, then $\ell=5$ and $G$ satisfies \crefWithTheorem{smallG}.
\end{lemma}

\begin{proof}
Suppose that $v_0 v_3$ and $v_{\ell-3} v_\ell$ cross.  This does not happen if $\ell \ge 6$, so $\ell = 5$ and the second chord is $v_2 v_5$.
Then $L_2=L_3=\emptyset$ by \cref{lem:endchord}.
If there is $w \in L_1$ then the paths $v_2 v_1 w$, $v_2 v_3 v_0$, $v_2 v_5 v_4$ form a $Y$.
The case $w \in L_4$ gives a $Y$ subgraph by symmetric arguments.
Hence $\Ll = \emptyset$ and \crefWithTheorem{smallG} holds.
\end{proof}

\begin{lemma}\label{lem:2chord3chord}
If a $2$-chord of $P$ crosses a $3$-chord $v_0v_3$ or $v_{\ell-3} v_\ell$ of $P$, then $\ell=5$ and $G$ satisfies \crefWithTheorem{smallG}.
\end{lemma}

\begin{proof}
Suppose the $3$-chord is $v_0v_3$; the argument for $v_{\ell-3} v_\ell$ is symmetric.
Then the $2$-chord is $v_2v_4$.
If $\ell\ge 6$, then $G$ has a $Y$ centered at $v_4$, so we may assume that $\ell = 5$.
Then $v_0v_5, v_1v_5 \notin E(G)$ by the assumptions of this section.  If $v_2 v_5 \in E(G)$, then \crefWithTheorem{smallG} holds by \cref{lem:3chordcross}, so we may assume that $v_2 v_5 \notin E(G)$.
\cref{lem:endchord,lem:2chordpendant} imply that $L_2=L_3=\emptyset$.
If $w \in L_1$, then $v_2 v_1 w$, $v_2 v_3 v_0$, and $v_2 v_4 v_5$ form a $Y$ subgraph, so $L_1 = \emptyset$.
If $L_4 = \emptyset$, then \crefWithTheorem{smallG} applies.  If $L_4 \ne \emptyset$ then $v_3 v_5 \notin E(G)$ by \cref{lem:endchord}, so all elements of $L_4$ are clones of the leaf $v_5$, and \crefWithTheorem{smallG} also applies.
\end{proof}

A $3$-chord and a vee do not cross.

\begin{lemma}\label{lem:3chordcrossvee}
A $3$-chord $v_0 v_3$ or $v_{\ell-3} v_\ell$ of $P$ does not cross a vee of $P$. 
\end{lemma}

\begin{proof}
Suppose the $3$-chord $v_0 v_3$ crosses a vee; the argument for $v_{\ell-3} v_\ell$ is symmetric.  The vee must have the form $v_2 w v_4$.  
Then the paths $v_3 v_0 v_1$, $v_3 v_2 w$, $v_3 v_4 v_5$ form a $Y$, which is a contradiction.
\end{proof}

We can now complete the case where $G$ does not have a long cycle.

\begin{proposition}\label{prop:pathstructure}
Suppose $G$ is a connected graph with no $Y$ subgraph and $P = v_0 v_1 \dots v_\ell$ is a pointiest longest path in $G$.  Suppose also that $\ell \ge 5$, $G$ has no chord from $\{v_0, v_1\}$ to $\{v_{\ell-1}, v_\ell\}$, and $L_1 \cap L_{\ell-1} = \emptyset$.
Then $G$ either satisfies \crefWithTheorem{smallG}, or $G$ is a spiked strand and satisfies \crefWithTheorem{infam}.
\end{proposition}

\begin{proof}
By the assumptions of this theorem and \cref{lem:domset,lem:k0klempty,lem:intersectKiKi1,lem:nonemptykikj,lem:nonemptyintersect3}, every edge of $G$ is one of the following: an edge of $P$, a chord of $P$, a pendant edge incident with $v_i$ where $1 \le i \le \ell-1$, or an edge of a unique vee $v_i w v_j$ with $1 \le i < j \le \ell-1$.  By \cref{lem:pos3pchords}, every chord is a $2$-chord or a $3$-chord (specifically $v_0 v_3$ or $v_{\ell-3} v_\ell$).

If $G$ has order $6$, then $G$ satisfies \crefWithTheorem{smallG}.  If $G$ has three consecutive $2$-chords, or a $3$-chord that crosses a $2$-chord or another $3$-chord, then by \cref{lem:3con2chords,lem:2chord3chord,lem:3chordcross}, $G$ also satisfies \crefWithTheorem{smallG}.
We may therefore assume that $G$ has order at least $7$, no three consecutive $2$-chords, and no $3$-chord that crosses a $2$-chord or another $3$-chord.
By \cref{lem:3chordcrossvee}, no $3$-chord crosses a vee.
By \cref{lem:vwithpend}, ($\ast$) all neighbors of vertices enclosed by vees lie on $P$, which includes the fact that vees do not cross.
By \cref{lem:2chordpendant}, ($\ast\ast$) all neighbors of vertices enclosed by $2$-chords lie on $P$, which includes the fact that a vee does not cross a $2$-chord.

Facts ($\ast$) and ($\ast\ast$) will be used later in the proof, and will be replaced by similar results when we use the same argument as part of the proof of \cref{prop:cyclestructure}.  These facts imply that two vees that cover the same edge, or a vee and a $2$-chord that cover the same edge, have the same endpoints.

Therefore, we can partition the edges of $P$ into subpaths of length at most three as follows.  For each $3$-chord we take (a) the triple of edges covered by the $3$-chord.  Some of these edges may be covered by vees or other $2$-chords, but there are no vees or other chords that cross the $3$-chord.  Edges of $P$ not covered by a $3$-chord can be divided into (b) triples of edges covered by either of a pair of crossing (equivalently, consecutive) $2$-chords, (c) pairs of edges covered by a $2$-chord and possibly also by vees, (d) pairs of edges covered only by one or more vees, and (e) single other edges (not covered by any chords or vees).  
We show that all other non-pendant edges of $G$ combine with one of these subpaths to form a bead where the ends of the subpath are the primary vertices of the bead, except for subpaths of type (a).  Moreover, the pendant edges not in $P$ form spikes, and so we have a spiked strand.

For subpaths $v_i v_{i+1} v_{i+2} v_{i+3}$ of type (b), ($\ast\ast$) implies that $L_{i+1} = L_{i+2} = \emptyset$.
Moreover, no additional chords cross $v_i v_{i+2}$ or $v_{i+1} v_{i+3}$, so $N(v_{i+1}) = \{v_i, v_{i+2}, v_{i+3}\}$ and $N(v_{i+2}) = \{v_i, v_{i+1}, v_{i+3}\}$.
The subpath $v_i v_{i+1} v_{i+2} v_{i+3}$ and the two $2$-chords form a $K_{2,1,1}$ bead with primary vertices $v_i$ and $v_{i+3}$.

For subpaths $v_i v_{i+1} v_{i+2}$ of type (c), $L_{i+1} = \emptyset$ by ($\ast\ast$).  There are also no chords crossing $v_i v_{i+2}$, so $N(v_{i+1}) = \{v_i, v_{i+2}\}$.
The subpath, the $2$-chord, and any vees that also cover the subpath form a $K_{1,1,t}$ bead where $t=|L_i\cap L_{i+2}|+1 \ge 1$, and $v_i$ and $v_{i+2}$ are primary vertices.

For subpaths $v_i v_{i+1} v_{i+2}$ of type (d), $N(v_{i+1}) = \{v_i, v_{i+2}\}$ by ($\ast$).  This subpath and the vees that cover it form a $K_{2,t}$ bead where $t=|L_i\cap L_{i+2}|+1 \ge 2$, with primary vertices $v_i$ and $v_{i+2}$.

Subpaths $v_i v_{i+1}$ of type (e) form a $K_{1,1,0}$ bead by themselves, with primary vertices $v_i$ and $v_{i+1}$.

The most complicated case is subpaths of type (a), which must be $v_0 v_1 v_2 v_3$ or $v_{\ell-3} v_{\ell-2} v_{\ell-1} v_\ell$.  The situations are symmetric, so we consider only the first, which is covered by the chord $v_0 v_3$.  By \cref{lem:k0klempty,lem:endchord}, $L_0 = L_2 = \emptyset$ and $L_1$ contains no leaves, so $L_1 \subseteq L_3$.  Consider the subgraph $H$ induced by $\{v_0, v_1, v_2, v_3\}$, which contains the $4$-cycle $C = (v_0 v_1 v_2 v_3)$.  If $H = C$ then $C$ and all vees of the form $v_1 w v_3$ form a $K_{2,t}$ bead with $t = |L_1 \cap L_3| + 2 \ge 2$ and primary vertices $v_1$ and $v_3$.  If $H = C \cup \{v_1 v_3\}$ then we similarly have a $K_{1,1,t}$ bead with primary vertices $v_1$ and $v_3$.   Suppose now that $v_0 v_2 \in E(H)$.  If there is $w \in L_1 = L_1 \cap L_3$ then the paths $v_3 w v_1$, $v_3 v_2 v_0$, $v_3 v_4 v_5$ form a $Y$ subgraph, so we must have $L_1 = \emptyset$.  If $H = C \cup \{v_0 v_2\}$ then $H$ is a $K_{2,1,1}$ bead with primary vertices $v_1$ and $v_3$, and if $H = C \cup \{v_0 v_2, v_1 v_3\}$ then $H$ is a $K_4$ bead with primary vertex $v_3$.  In all situations, since we have accounted for all edges from $v_i$ to $L_i$ for $0 \le i \le 2$, and there are no chords crossing $v_0 v_3$, the bead is joined to the rest of $G$ only at the primary vertex $v_3$.

The above analysis accounts for all edges of $P$ and non-pendant edges not in $P$.  It shows moreover that any pendant edges not in $P$ must be adjacent to primary vertices; such primary vertices must belong to two beads or we could find a path longer than $P$.  Therefore, we have a spiked strand.
\end{proof}
 
\section{Long cycle}\label{Sec:CycleCase}
We now consider the case that $P$ has a chord, or an intersection between $L_1$ and $L_{\ell-1}$, that creates a long cycle.  As in \cref{Sec:PathCase}, we assume that $P$ is a pointiest longest path $v_0 v_1 \dots v_{\ell}$ with $\ell \ge 5$.  But now we assume that either $P$ has a chord from $v_0$ or $v_1$ to $v_{\ell-1}$ or $v_\ell$, or $L_1 \cap L_{\ell-1} \ne \emptyset$.  
We first show that there is a long cycle that is edge-dominating.

\begin{lemma}\label{lem:edgedomcycle}
The graph $G$ has an edge-dominating cycle of length at least $\ell-1$, and hence at least $4$.
\end{lemma}

\begin{proof}
By \cref{lem:domset}, $P$ is an edge-dominating path in $G$, so if $v_0 v_\ell \in E(G)$, then $(v_0 v_1 \dots v_\ell)$ is an edge-dominating cycle of length $\ell+1$.  If $v_0 v_\ell \notin E(G)$, then
$v_1 v_2 \dots v_{\ell-1}$ is an edge-dominating path in $G$, by \cref{lem:domset,lem:k0klempty}.  Therefore, any cycle that contains this path is an edge-dominating cycle of length at least $\ell-1$.  If we have a chord $v_i v_j$ with $(i,j) = (0,\ell-1)$, $(1, \ell-1)$, or $(1, \ell)$, then $(v_i v_{i+1} \dots v_j)$ is such a cycle.  If we have $w \in L_1 \cap L_{\ell-1}$, then $(w v_1 v_2 \dots v_{\ell-1})$ is such a cycle.
\end{proof}

Our arguments in the rest of this section rely only on the hypothesis that $G$ is a connected graph with no subgraph isomorphic to $Y$ and with an edge-dominating cycle of length at least $4$.
We choose $C = (u_0 u_1 u_2 \dots u_{k-1})$ to be an edge-dominating cycle in $G$ of maximum length $k \ge 4$.  The subscript $i$ of $u_i$ is considered to belong to $\mZ_k = \{0, 1, 2, \dots, k-1\}$ with operations modulo $k$.  For $i, j \in \mZ_k$ we let $d(i,j)$ denote the cyclic distance between $i$ and $j$, i.e., if $i \le j$ then $d(i,j) = \min(j-i, k+i-j)$ and if $i > j$ then $d(i,j) = d(j,i)$.
Let $N_i=N_G(u_i)-V(C)$ and $\N = \bigcup_{i=0}^{k-1} N_i$.

We prove a series of lemmas similar to those in \cref{Sec:PathCycle,Sec:PathCase}.  The arguments here are sometimes simpler since we do not have to worry about ``end effects''.  Since the arguments are very similar we omit some details.
We focus on the relationships between the sets $N_i$ and on the existence of particular chords of the cycle  $C$.  Henceforth (for the rest of this section) an \emph{$i$-chord} is an edge $u_j u_m$ with $d(j,m)=i$, and a \emph{vee} is a path $u_j w u_m$ where $d(j, m) = 2$ and $w \in N_j \cap N_m$.  We say that a chord $u_j u_m$ \emph{crosses} another chord $u_p u_q$ if $j, p, m, q$ are distinct and occur along $C$ in either that cyclic order or its reverse.  We similarly define what it means for two vees to cross, or for a vee to cross a chord.

Because $C$ is edge-dominating we have the following.

\begin{observation}\label{obs:domcycle}
We have $V(G)-V(C) = \mathcal{N}$, and $\mathcal{N}$ is an independent set in $G$.
\end{observation}

Next we consider intersections between the sets $N_i$.
\begin{lemma}\label{lem:nini1empty}
For $i\in\mZ_k$,  $N_i\cap N_{i+1}=\emptyset$.
\end{lemma}
\begin{proof}
If $w \in N_i\cap N_{i+1}$, then replacing $u_i u_{i+1}$ by $u_i w u_{i+1}$ gives an edge-dominating cycle longer than $C$.  Thus, $N_i\cap N_{i+1}=\emptyset$, as required.
\end{proof}

\begin{lemma}\label{lem:intersect2Ni}
If $N_i\cap N_j\ne \emptyset$ where $i \ne j$ then $d(i,j) = 2$.
\end{lemma}
\begin{proof}
Suppose that $N_i \cap N_j \ne \emptyset$.
By \cref{lem:nini1empty}, $d(i,j) \ge 2$ and hence $k \ge 4$.  If $k \le 5$ we have $d(i,j) = 2$ as required.
We may therefore assume that $k\ge 6$.
If $d(i,j) \ge 3$ and $w \in N_i \cap N_j$, then the paths $u_i w u_j$, $u_i u_{i-1} u_{i-2}$, $u_i u_{i+1} u_{i+2}$ form a $Y$ subgraph, a contradiction.
The conclusion follows.
\end{proof}

Each edge in $E(G)-E(C)$ belongs to at most one vee.

\begin{lemma}\label{lem:intersect3Ni}
If $i$, $j$, and $m$ are distinct, then $N_i\cap N_j\cap N_m=\emptyset$.
\end{lemma}
\begin{proof}
Suppose there is $w \in N_i \cap N_j \cap N_m$.  By \cref{lem:intersect2Ni}, $d(i,j)=d(j,m)=d(i,m) = 2$.  This can only happen when $k=6$, and then there is a $Y$ subgraph centered at $w$, which is a contradiction.
\end{proof}

\begin{lemma}\label{lem:cycleveecross}
A vee does not cross a $2$-chord or another vee.
\end{lemma}

\begin{proof}
If either crossing occurred we would have an edge-dominating cycle longer than $C$, which is a contradiction.
\end{proof}

We deal with the smallest value of $k$ separately. 

\begin{lemma}\label{lem:cyclek4}
If $k=4$, then $G$ satisfies \crefWithTheorem{smallG}, or $G$ is a spiked necklace and satisfies \crefWithTheorem{infam}.
\end{lemma}
\begin{proof}
By \cref{lem:nini1empty} we have $N_i\cap N_{i+1}=\emptyset$ for $i \in \mZ_4$. 
By \cref{lem:cycleveecross} none of the following occur: $N_0\cap N_2\ne \emptyset$ and $N_1\cap N_3\ne \emptyset$; $N_0\cap N_2\ne \emptyset$ and $u_1u_3 \in E(G)$; or $N_1\cap N_3\ne \emptyset$ and $u_0u_2 \in E(G)$.

Suppose that $G$ has a vee.  By symmetry, we may assume that $N_0 \cap N_2 \ne \emptyset$; then by \cref{lem:cycleveecross}, $N_1 \cap N_3 = \emptyset$ and $u_1 u_3 \notin E(G)$.  If $N_1 \ne \emptyset$ and $N_3 \ne \emptyset$, then $G$ has a $Y$ subgraph.
So at least one of $N_1, N_3$ is empty: without loss of generality assume that $N_3 = \emptyset$.
If $u_0 u_2 \in E(G)$ then $G$ may be regarded as a spiked necklace with three beads.  Two are $K_{1,1,0}$ beads $u_0 u_1$, $u_1 u_2$ and the other is a $K_{1,1,t}$ bead with primary vertices $u_0$ and $u_2$, formed by $u_0 u_2$, $u_0 u_3 u_2$ and all vees $u_0 w u_2$, with $t = |N_0 \cap N_2| + 1 \ge 2$.  The edges $w u_1$ for $w \in N_1$ are spikes.
If $u_0 u_2 \notin E(G)$ then in a similar way $G$ may be regarded as a spiked necklace with two $K_{1,1,0}$ beads and a $K_{2,t}$ bead, $t \ge 2$.  We may therefore assume that $G$ does not have a vee.

If $G$ has no chords, then $G$ is a cycle with potential pendant edges at each vertex, which is a spiked necklace.
If $G$ has exactly one chord, we may suppose without loss of generality that it is $u_0u_2$.
If $G$ has pendant edges attached to at most two vertices of $C$, then $G$ satisfies \crefWithTheorem{smallG}.
If $G$ has pendant edges at $u_1$, $u_3$, and at least one of $u_0$ or $u_2$, then $G$ has a $Y$ subgraph, a contradiction.
If $G$ has pendant edges at $u_0$, $u_2$, and one of $u_1$ or $u_3$, then it is a spiked necklace with two $K_{1,1,0}$ beads and a $K_{1,1,1}$ bead.
We may therefore assume that $G$ has two chords.

Then $G$ has pendant edges attached to at most two vertices in $C$, otherwise $G$ has a $Y$.
Thus, $G$ satisfies \crefWithTheorem{smallG}.
\end{proof}

We henceforth (for the rest of this section) assume that $k\ge 5$.  Because $k \ge 5$, if we have a vee then there is a unique subpath of $C$ of length $2$ joining the endpoints of the vee.  We say the middle vertex of this path is \emph{enclosed} by the vee, and the edges of this path are \emph{covered} by the vee.  We can similarly define the enclosed vertex and covered edges for a $2$-chord.
While vees do not cross $2$-chords or other vees, there may be multiple vees that cover the same pair of edges, and vees may cover the same pair of edges as a $2$-chord.

We now show that all neighbors of the enclosed vertex of a vee are on $C$.

\begin{lemma}\label{lem:cyclependv}
Suppose $i \in \mZ_k$.  Then we do not have $N_i \ne \emptyset$ and $N_{i-1} \cap N_{i+1} \ne \emptyset$.
\end{lemma}
\begin{proof}
Suppose that $w \in N_i$ and $v \in N_{i-1} \cap N_{i+1}$.  Then $w \ne v$ by \cref{lem:intersect2Ni}.  Since $k \ge 5$, the paths $u_{i-1} u_{i-2} u_{i-3}$, $u_{i-1} u_i w$, and $u_{i-1} v u_{i+1}$ form a $Y$ subgraph, which is a contradiction.
\end{proof}

Except in some situations covered by \crefWithTheorem{smallG}, all chords are $2$-chords.

\begin{lemma}\label{lem:cycleboundchords}
If there is an $i$-chord for $i\ge 3$, then $k=6$ and $G$ satisfies \crefWithTheorem{smallG}.
\end{lemma}
\begin{proof}
Suppose there is an $i$-chord $e$ for $i \ge 3$.
If $k=5$, this is not possible.
If $k=6$, then $e$ is a $3$-chord.  If $V(G) = V(C)$ then \crefWithTheorem{smallG} holds, so we may assume there is at least one vertex $w$ not on $C$.  Then in each of the two possible cases ($w$ is or is not adjacent to an end of $e$) there is a $Y$ subgraph in $G$, which is a contradiction.  We may therefore assume that $k \ge 7$.
Then $G$ has a $Y$ centered at either endpoint of $e$, a contradiction.
Therefore, there is no such $i$-chord.
\end{proof}

Except in some situations covered by \crefWithTheorem{smallG}, all neighbors of the enclosed vertex of a $2$-chord are on $C$.

\begin{lemma}\label{lem:cyclependchord}
If $u_{i-1} u_{i+1} \in E(G)$ and $N_i \ne \emptyset$, then $k=5$ and $G$ satisfies \crefWithTheorem{smallG}.
\end{lemma}
\begin{proof}
Suppose that $u_{i-1}u_{i+1} \in E(G)$, and assume there is $v \in N_i$.  Without loss of generality we may assume that $i=1$.  If $k\ge 6$, then $G$ contains a $Y$ subgraph centered at $u_0$, a contradiction.

Suppose now that $k = 5$.  All vertices in $N_1$ are leaves by \cref{lem:cycleveecross}.
If $\N = N_1$ then \crefWithTheorem{smallG} holds, so we may assume that there is $w \in N_j$ for some $j \ne 1$.  There are two possible cases $j = 2, 3$ (up to symmetry) and in both cases there is a $Y$ subgraph, a contradiction.
\end{proof}

Except in some situations covered by \crefWithTheorem{smallG}, $C$ does not have three consecutive $2$-chords.

\begin{lemma}\label{lem:cycle3conschords}
If there are three $2$-chords $u_i u_{i+2}$, $u_{i+1} u_{i+3}$ and $u_{i+2} u_{i+4}$, then $k \le 6$ and $G$ satisfies \crefWithTheorem{smallG}.
\end{lemma}
\begin{proof}
Suppose three such $2$-chords exist; without loss of generality we may assume they are $u_0 u_2$, $u_1 u_3$ and $u_2 u_4$.
If $k\ge 7$, then $G$ has a $Y$ centered at $u_2$.
If $k=6$ then there must be $v \in N_j$ for some $j$, otherwise $G$ satisfies \crefWithTheorem{smallG}.  There are four possible cases $j = 2, 3, 4, 5$ (up to symmetry) and in each case $G$ has a $Y$ subgraph.

Suppose now that $k=5$.  If $|\N| \le 1$, or $\N = N_j$ for some $j$ and all elements of $N_j$ are leaves, then \crefWithTheorem{smallG} again holds.  So we may assume that $|\N| \ge 2$, with $N_j \ne \emptyset$, and not all elements of $\N$ are leaves in $N_j$.  Thus, if $\N = N_j$, then there is a non-leaf $w \in N_j$, and also $v \in N_j-\{w\}$.  If $\N \ne N_j$, then there is $w \in \N-N_j$ and also $v \in N_j$.  In either case we have $v \in N_j$ and $w \in N_m$ for some $m \ne j$, with $v \ne w$.
The cycle $C$ and the three $2$-chords form a wheel centered at $u_2$ with rim $R=(u_0 u_1 u_3 u_4)$.
Thus, there are three cases up to symmetry: $j=2$ and $m \ne 2$; $u_j$ and $u_m$ are adjacent in $R$; and $u_j$ and $u_m$ are opposite in $R$.  In each case there is a $Y$ subgraph.
\end{proof}

We can now complete the case where $G$ has a long cycle.  Note that we have counterparts for all of the results in \cref{Sec:PathCycle,Sec:PathCase} that deal with the existence and behavior of $2$-chords and vees and neighbors of their enclosed vertices, so we can use the same arguments as in \cref{prop:pathstructure}.

\begin{proposition}\label{prop:cyclestructure}
Suppose $G$ is a connected graph with no $Y$ subgraph and with an edge-dominating cycle $C$ of length at least $4$.  Then $G$ either satisfies \crefWithTheorem{smallG}, or $G$ is a spiked necklace and satisfies \crefWithTheorem{infam}.
\end{proposition}

\begin{proof}
If $k=4$, then the result holds by \cref{lem:cyclek4}.  We may therefore assume that $k\ge 5$.
By \cref{lem:edgedomcycle,obs:domcycle,lem:nini1empty,lem:intersect2Ni,lem:intersect3Ni}, 
every edge of $G$ is one of the following: an edge of $P$, a chord of $P$, a pendant edge incident with some vertex of $P$, or an edge of a unique vee.

By \cref{lem:cyclependv} ($\ast$) all neighbors of vertices enclosed by vees lie on $C$, which includes the fact that vees do not cross.
If there is a chord that is not a $2$-chord, or the enclosed vertex of a $2$-chord has a neighbor not on $C$, or there are three consecutive $2$-chords, then by \cref{lem:cycleboundchords,lem:cyclependchord,lem:cycle3conschords} $G$ satisfies \crefWithTheorem{smallG}.
We may therefore assume that all chords are $2$-chords, ($\ast\ast$) the neighbors of a vertex enclosed by a $2$-chord lie on $C$, which includes the fact that a vee does not cross a $2$-chord, and there are no three consecutive $2$-chords.

Facts ($\ast$) and ($\ast\ast$) imply that two vees that cover the same edge, or a vee and a $2$-chord that cover the same edge, have the same endpoints.  These facts will be used below when we apply an argument from the proof of \cref{prop:pathstructure}.

Therefore, we can partition the edges of $C$ into subpaths of length at most three in a way similar to the proof of \cref{prop:pathstructure}.  We no longer have subpaths of type (a), since we no longer have $3$-chords.  So we have the following: (b) triples of edges covered by either of a pair of crossing $2$-chords, (c) pairs of edges covered by a $2$-chord and possibly also by vees, (d) pairs of edges covered only by one or more vees, and (e) single other edges (not covered by any chords or vees).  We apply arguments very similar to the proof of \cref{prop:pathstructure}, replacing $v_i$ by $u_i$ and $L_i$ by $N_i$, and using ($\ast$) and ($\ast\ast$) from above. We conclude that either all other non-pendant edges of $G$ combine with one of these subpaths to form a bead whose primary vertices are the ends of the subpath, and that the pendant edges form spikes, so we have a spiked necklace.
\end{proof}

\section{Main result}\label{Sec:MainResult}
We now restate and prove \cref{thm:noY}.

\structthm*

\begin{proof}
First suppose that $G$ is a graph that does not have $Y$ as a subgraph.  Let $P=v_0 v_1 \dots v_\ell$ be a pointiest longest path in $G$.  If $\ell \le 4$ then the result holds by \cref{prop:small}.  If $\ell \ge 5$ and $P$ does not have any chords from $\{v_0, v_1\}$ to $\{v_{\ell-1}, v_\ell\}$, and $L_1 \cap L_{\ell-1} = \emptyset$, then the result holds by \cref{prop:pathstructure}.  Otherwise, the result holds by \cref{lem:edgedomcycle,prop:cyclestructure}.

Conversely, suppose that $G$ satisfies \crefWithTheorem{smallG} or \crefWithTheorem{infam}.
Assume that $G$ contains $Y$ as a subgraph.
If $G$ satisfies \crefWithTheorem{smallG}, then at most one vertex in each leaf class can be used in the $Y$.
Therefore if we delete all but one vertex in each leaf class to obtain a graph $G'$, then $G'$ still contains $Y$ as a subgraph.
This is a contradiction because $G'$ has at most six vertices.

Now suppose that $G$ satisfies \crefWithTheorem{infam}.
Then the $Y$ subgraph may be centered at a primary vertex or at a secondary vertex of degree at least $3$.
A secondary vertex $v$ of degree at least $3$ occurs only in beads that are $K_4$ or $K_{2,1,1}$.
In either case, $v$ has degree exactly $3$, and a $Y$ cannot be centered at $v$ because 
once the three edges incident with $v$ are used, at least one of the paths cannot be further extended.

We may therefore assume that the $Y$ is centered at a primary vertex $v$.
Pendant edges at $v$ cannot be used in the $Y$ subgraph.
In order to obtain a $Y$ subgraph, there would have to be two paths of length $2$ beginning at the primary vertex $v$ and disjoint except at $v$, that are contained in, or pass through, the same bead containing $v$.  However, no bead allows this.
Hence $G$ does not have $Y$ as a subgraph, as required.
\end{proof}

\section{Forbidding subtrees of $Y$}\label{sec:Subtree}
There are seven proper subtrees of $Y$ with at least one edge. The following results characterize graphs that do not have each of those subtrees.  We include some trivial observations for completeness, without proof.

We will fix a subtree $T$ of $Y$ and consider a connected graph $G$ that has no $T$ subgraph.  We assume that $P = v_0 v_1 \dots v_\ell$ is a pointiest longest path in $G$.  Since $G$ also has no $Y$ subgraph, all results in \cref{Sec:PathCycle} also apply to $G$ and $P$, and we adopt the notation and terminology of \cref{Sec:PathCycle}.  In particular, we often tacitly use the fact that every edge of $G$ is incident with a vertex of $P$ ($P$ is edge-dominating).

We begin by considering the subtrees $T$ of $Y$ that are paths.  For $k \le 5$, having no $P_k$ as a subgraph is equivalent to having no $Y$ as a subgraph and having $\ell \le k-2$, so we can just use results from \cref{Sec:PathCycle}.

\begin{observation}\label{obs:forbidP2}
A connected graph does not have $P_2$ as a subgraph (or minor) if and only if it is $K_1$.
\end{observation}

\begin{observation}\label{obs:forbidP3}
A connected graph does not have $P_3$ as a subgraph (or minor) if and only if it is $K_1$ or $K_2$.
\end{observation}

\begin{theorem}\label{theorem:forbidP4}
A connected graph does not have $P_4$ as a subgraph (or minor) if and only if it is a triangle or a star. 
\end{theorem}

\begin{proof}
This follows from \cref{obs:forbidP2,obs:forbidP3}, and \cref{lem:lengthl2}.
\end{proof}

\begin{theorem}\label{theorem:forbidP5}
A connected graph does not have $P_5$ as a subgraph (or minor) if and only if it is a star, a double star, a graph of order at most $4$, or a graph obtained by attaching pendant edges to exactly one vertex of a triangle. 
\end{theorem}

\begin{proof}
This follows from \cref{obs:forbidP2,obs:forbidP3}, and \cref{lem:lengthl2,lem:lengthl3}.
\end{proof}

Now we consider subtrees of $Y$ that have a vertex of degree $3$.

\begin{observation}\label{thm:forbidClaw}
A connected graph does not contain a claw $K_{1,3}$ as a subgraph (or minor) if and only if it is a path or cycle.
\end{observation}

Let $Y_1$ be the graph obtained from $P_4=x_0 x_1 x_2 x_3$ by attaching a pendant edge to $x_2$, as in \cref{fig:Y1}.  
\begin{figure}[h]
\begin{center}
     \begin{tikzpicture}
        [scale=1,auto=left,every node/.style={circle, fill, inner sep=0 pt, minimum size=1.75mm, outer sep=0pt},line width=.4mm]
        \node (0) at (0,0) [label={[label distance=2pt]below:$x_0$}]{};
        \node (1) at (1,0) [label={[label distance=2pt]below:$x_1$}]{};
        \node (2) at (2,0) [label={[label distance=2pt]below:$x_2$}]{};
        \node (3) at (3,0) [label={[label distance=2pt]below:$x_3$}]{};
        \node (4) at (2,1) {};
        \draw (0) to (1) to (2) to (3);
        \draw (2) to (4);
	\end{tikzpicture}
\end{center}
\caption{$Y_1$}
\label{fig:Y1}
\end{figure}

\begin{theorem}\label{thm:forbidY1}
A connected graph does not have $Y_1$ as a subgraph (or minor) if and only if it is  one of the following:
a graph of order at most $4$, a star, a path, or a cycle.
\end{theorem}

\begin{proof}
The listed graphs obviously do not have $Y_1$ as a subgraph.

Suppose now that $G$ is a connected graph with no $Y_1$ subgraph and order at least $5$.
If $\ell \le 2$ then $G$ is a star, by \cref{obs:forbidP2,obs:forbidP3}, and \cref{lem:lengthl2}.  If $\ell = 3$ then $G$ is a double star or a triangle with pendant edges at one vertex, by \cref{lem:lengthl3}, but all such graphs with order at least $5$ contain a $Y_1$ subgraph.
So we may assume that $\ell \ge 4$.

If $\Ll \ne \emptyset$, then $L_i \ne \emptyset$ for some $i$ with $1 \le i \le \ell-1$, and $G$ has a $Y_1$ subgraph.  So we may assume that $\Ll = \emptyset$, i.e., that all edges in $E(G)-E(P)$ are chords of $P$.  
If $G$ has no chord, or the only chord is $v_0 v_\ell$, then $G$ is a path or cycle.  So we may assume that there is a chord $v_i v_j$, $i < j$, with $i \ne 0$ or $j \ne \ell$.

Suppose that $j \ge i+3$.  If $i \ne 0$ then the paths $v_{i-1} v_i v_{i+1}$ and $v_i v_j v_{j-1}$ form a $Y_1$ subgraph, and a symmetric argument applies if $j \ne \ell$.

So we may assume that $j = i+2$, where $0 \le i \le \ell-2$.
If $i=0$, then $v_0 v_2 v_1$, $v_2 v_3 v_4$ form a $Y_1$ subgraph, and a symmetric argument applies if $j = \ell$, i.e., $i = \ell-2$.
If $1 \le i \le \ell-3$ then $v_{i-1} v_i v_{i+1}$, $v_i v_{i+2} v_{i+3}$ form a $Y_1$ subgraph.
Thus, we obtain no graphs beyond those in the statement of the theorem.
\end{proof}

Let $Y_2$ be the graph obtained from $P_5=x_0x_1x_2x_3x_4$ by attaching a pendant edge to $x_2$, as in \cref{sub@fig:Y2}. Let $\TS$ be the family of graphs obtained by taking a path of length at least $1$ and at each endpoint attaching either a triangle or the center of a star $K_{1,t}$ with $t \ge 1$.  An example appears in \cref{sub@fig:TS}. 
\begin{figure}[h!]
\begin{subfigure}[b]{.47\textwidth}
\begin{center}
     \begin{tikzpicture}
        [scale=1,auto=left,every node/.style={circle, fill, inner sep=0 pt, minimum size=1.75mm, outer sep=0pt},line width=.4mm]
        \node (0) at (0,0) [label={[label distance=2pt]below:$x_0$}]{};
        \node (1) at (1,0) [label={[label distance=2pt]below:$x_1$}]{};
        \node (2) at (2,0) [label={[label distance=2pt]below:$x_2$}]{};
        \node (3) at (3,0) [label={[label distance=2pt]below:$x_3$}]{};
        \node (4) at (4,0) [label={[label distance=2pt]below:$x_4$}] {};
        \node (5) at (2,1){};
        \draw (0) to (1) to (2) to (3) to (4);
        \draw (2) to (5);
	\end{tikzpicture}
\end{center}
\caption{$Y_2$}
\label{fig:Y2}
\end{subfigure}
\begin{subfigure}[b]{.47\textwidth}
\begin{center}
     \begin{tikzpicture}
        [scale=1,auto=left,every node/.style={circle, fill, inner sep=0 pt, minimum size=1.75mm, outer sep=0pt},line width=.4mm]
        \node (0) at (0,0) {};
        \node (a) at (0,.5) {};
        \node (b) at (0,1) {};
        \node (1) at (1,.5){};
        \node (2) at (2,.5) {};
        \node (3) at (3,.5) {};
        \node (4) at (4,1) {};
        \node (5) at (4,0) {};
        \draw (0) to (1) to (2) to (3) to (4) to (5) to (3);
        \draw (a) to (1) to (b);
	\end{tikzpicture}
\end{center}
\caption{A member of the family $\TS$}
\label{fig:TS}
\end{subfigure}
\end{figure}

\begin{theorem}\label{thm:forbidY2}
A connected graph does not have $Y_2$ as a subgraph (or minor) if and only if it is one of the following:
a path, a cycle, a star, a member of $\TS$, a $K_{1,1,t}$ ($t \ge 2$) or $K_{2,t}$ ($t \ge 2$) bead with possibly pendant edges attached at primary vertices, or a graph obtained from a graph of order at most $5$ by cloning leaves.
\end{theorem}
\begin{proof}
The listed graphs do not have $Y_2$ as a subgraph.
Cloning leaves does not increase $\alpha'$, so if $G$ is obtained from a graph $G'$ of order at most $5$ by cloning leaves, then $\alpha'(G) = \alpha'(G') \le 2$.  But $\alpha'(Y_2) = 3$, so $Y_2$ is not a subgraph of $G$.
In the other cases, it is clear that the center vertex of any path $Q$ of length $4$ has no neighbor outside of $Q$, so $Q$ cannot be contained in a $Y_2$ subgraph, and hence $G$ has no $Y_2$ subgraph.

Suppose now that $G$ is a connected graph with no $Y_2$ subgraph and order at least $6$.
If $\ell \le 2$ then $G$ is a star, by \cref{obs:forbidP2,obs:forbidP3}, and \cref{lem:lengthl2}.  If $\ell = 3$, then $G$ is a double star or a triangle with pendant edges at one vertex, by \cref{lem:lengthl3}. Thus, $G$ is obtained from a graph of order at most $4$ by cloning leaves.
So we may assume that $\ell \ge 4$.  The following general claims hold in this situation.

\setcounter{claim}{0}

\begin{claim}\label{clm:b0} If $L_i \ne \emptyset$ then $i = 1$ or $\ell-1$.  By \cref{lem:k0klempty} we have $L_0 = L_\ell = \emptyset$, and we have a $Y_2$ subgraph if $L_i \ne \emptyset$ for $2 \le i \le \ell-2$.
\end{claim}

\begin{claim}\label{clm:b2}
By \cref{lem:endchord}, none of the following happen: (a) $L_1 \ne \emptyset$ and $v_0 v_2 \in E(G)$, (b) there is a leaf in $L_1$ and $v_0 v_3 \in E(G)$, (c) $L_{\ell-1} \ne \emptyset$ and $v_{\ell-2} v_\ell \in E(G)$, (d) there is a leaf in $L_{\ell-1}$ and $v_{\ell-3} v_3 \in E(G)$.
\end{claim}

Suppose now that $\ell = 4$, so by \cref{clm:b0} $L_0 = L_2 = L_4 = \emptyset$, and \cref{clm:b2} puts restrictions on the chords $v_0 v_2, v_0 v_3, v_1 v_4, v_2 v_4$.  The following also holds in this case.

\begin{claim}\label{clm:b1}
When $\ell=4$ we may assume that $v_0 v_4 \notin E(G)$.  If $v_0 v_4 \in E(G)$, there is a cycle $C$ with $V(C) = V(P)$, so if $\Ll \ne \emptyset$ there is a path longer than $P$, which is impossible, and otherwise $G$ has order $5$.
\end{claim}

Assume first that $L_1 \cap L_3 \ne \emptyset$.  By \cref{clm:b1,clm:b2}, $v_0 v_4, v_0 v_2, v_2 v_4 \notin E(G)$, so the only possible chords are $v_0 v_3$, $v_1 v_3$ and $v_1 v_4$.
If $v_1 v_3 \in E(G)$ there is a $K_{1,1,t}$ bead, $t \ge 2$, with primary vertices $v_1, v_3$ and secondary vertices including $v_2$, $L_1 \cap L_3$ and possibly $v_0$ (if $v_0 v_3 \in E(G)$) or $v_4$ (if $v_1 v_4 \in E(G)$); the remaining edges form at most two $K_{1,1,0}$ beads and spikes.  If $v_1 v_3 \notin E(G)$ we have a similar situation but with a $K_{2,t}$ bead, $t \ge 2$.  In either case we have a graph described in the theorem.

Assume now that $L_1 \cap L_3 = \emptyset$.  Then all vertices in $L_1 \cup L_3$ are leaves.  If $L_1 \ne \emptyset$ then by \cref{clm:b2,clm:b1} there are no chords incident with $v_0$.  So whether $L_1=\emptyset$ or not, all elements of $L_1$ are clones of $v_0$.  Similarly, all elements of $L_3$ are clones of $v_4$.  Therefore, $G$ is obtained from a graph of order $5$ by (possibly) cloning leaves.  This concludes the case $\ell = 4$.

Suppose now that $\ell \ge 5$.  \cref{clm:b0,clm:b2} apply, and \cref{lem:pos3pchords} describes the possible chords of $P$.
If $w \in L_1\cap L_{\ell-1}$, then the paths $v_1 v_0$, $v_3 v_2 v_1 w v_{\ell-1}$ form a $Y_2$ subgraph, so $L_1 \cap L_{\ell-1} = \emptyset$.  Hence all vertices in $L_1$ and $L_{\ell-1}$ are leaves.

If $P$ has the chord $v_0 v_\ell$ then $P \cup \{v_0v_\ell\}$ is a cycle of length $\ell+1 \ge 6$, and we must have $L_i = \emptyset$ for all $i$, otherwise we obtain a path longer than $P$.  Any chord in a cycle of length $6$ or more creates a $Y_2$ subgraph, so $G$ must be a cycle.
If $P$ has the chord $v_1 v_\ell$ then the paths $v_1 v_0$, $v_3 v_2 v_1 v_\ell v_{\ell-1}$ form a $Y_2$ subgraph; a symmetric argument applies for $v_0 v_{\ell-1}$.
If $P$ has the chord $v_1 v_{\ell-1}$ then the paths $v_1 v_0$, $v_3 v_2 v_1 v_{\ell-1} v_\ell$ form a $Y_2$ subgraph.
If $P$ has the chord $v_0 v_3$ then the paths $v_3 v_0$, $v_1v_2 v_3 v_4 v_5$ form a $Y_2$ subgraph; a symmetric argument applies for $v_{\ell-3} v_\ell$.
If $P$ has a chord $v_i v_j$ with $j = i+2$ and $1 \le i \le \ell-4$, then the paths $v_{i+2} v_{i+1}$, $v_{i-1} v_i v_{i+2} v_{i+3} v_{i+4}$ form a $Y_2$ subgraph; a symmetric argument applies if $4 \le j \le \ell-1$, i.e., $2 \le i \le \ell-3$.  At this point we have dealt with all possible chords except $v_0 v_2$ and $v_{\ell-2} v_\ell$.

If $v_0 v_2$ is a chord then $L_1 = \emptyset$ by \cref{clm:b2}, and we have a triangle $(v_0 v_1 v_2)$ attached to a path at $v_2$; if $v_0 v_2$ is not a chord then we have a star with center $v_1$ and leaves in the set $\{v_0\} \cup L_1$ attached to a path at $v_1$.  A similar argument applies at the other end of the path, so we have a graph in the family $\TS$ (which includes paths of length at least $3$).
\end{proof}

\section{Infinite versions}\label{sec:BoundedPathLength}
In this section, we consider infinite analogues of the results in \cref{Sec:MainResult,sec:Subtree}.

Our previous arguments based on pointiest longest paths and longest edge-dominating cycles still work in the case of possibly infinite graphs of \emph{bounded path length}, i.e., where there is some finite number $\ell$ so that all paths in $G$ have length at most $\ell$.  The only difference is that in situations where we could have many vertices that are clones (such as in $K_{1,1,t}$ or $K_{2,t}$ beads, or with spikes, or with cloning leaves in small graphs) we can now have infinitely many such vertices.  Some of the choices we made in \cref{Sec:PathCycle} were intended to make the arguments work easily for infinite graphs of bounded path length, including our proof of \cref{lem:domset}, and the way in which we defined pointiness.

The results in this section can be proved without using the Axiom of Choice, even in a weak form such as K\"{o}nig's Lemma.  We are careful to use arguments that do not require a choice principle.  We discuss this further in \cref{aoc} at the end of this section.

In the infinite case a strand or necklace is built by stringing together a finite number of (possibly infinite) beads into a necklace or strand, and we get a spiked necklace or strand by adding (possibly infinitely many) optional spikes incident with primary vertices that are in exactly two beads.  A graph is a member of the family $\BinfB$ if and only if it is a spiked strand or spiked necklace.  All graphs in $\BinfB$ are connected and have bounded path length.  \cref{thm:noY} extends straightforwardly to the following.

\begin{proposition}\label{prop:noYinfbd}
A possibly infinite connected graph $G$ with bounded path length has no subgraph isomorphic to $Y$ if and only if either
\begin{statement}
\item\label{infbdsmallG} $G$ is obtained from a graph with at most six vertices by adding (possibly infinitely many) optional clones of leaves, or 
\item\label{infbdinfam} $G$ belongs to the family $\B_\infty^{\fam0 B}$.
\end{statement}
\end{proposition}

The following obvious corollary will be used in the proof of \cref{prop:unboundedpathl}.

\begin{corollary}\label{cor:sn}
Every possibly infinite graph $G$ with maximum path length $\ell$,  $6 \le \ell < \infty$, and no $Y$ subgraph is a spiked strand or spiked necklace.
\end{corollary}

Now we consider the case where $G$ has unbounded path length.  First we show that every block in a graph with no $Y$ subgraph has bounded path length.

\begin{observation}\label{obs:pathblock}
Suppose a path $P$ in a graph $G$ intersects a block $B$ of $G$.  Then $P \cap B$ is a subpath of $P$.
\end{observation}

\begin{lemma}\label{lem:boundedpathl}
Suppose $G$ is an infinite graph with no $Y$ subgraph.  Then each block $B$ of $G$ has bounded path length, i.e., there is some finite $\ell$ such that every path in $B$ has length at most $\ell$.
Hence, each union of finitely many blocks of $G$ has bounded path length.
\end{lemma}

\begin{proof}
If $B$ is $K_1$ or $K_2$, then $B$ has bounded path length, so we may assume that $B$ is $2$-connected.  Suppose for a contradiction that $B$ has unbounded path length.

\setcounter{claim}{0}
\begin{claim}\label{clm:c1}
If $C$ is a cycle in $B$, then there exists a path $P$ in $G$ with at least two vertices such that $V(C) \cap V(P) = \emptyset$.  To prove this, suppose $|V(C)| = k$, and choose a path $Q$ with at least $2k+2$ vertices.  Then $Q-V(C)$ has at most $k+1$ components and at least $k+2$ vertices, so one of its components has at least two vertices, and is the desired $P$.
\end{claim}

We first show that $B$ has a cycle $C_1$ of length at least $5$.  Choose a cycle $C_0$ in $B$, which exists because $B$ is $2$-connected.  If $C_0$ has length at least $5$, let $C_1=C_0$.  Otherwise, $C_0$ has length $3$ or $4$, and so $C_0$ contains a path of length at least $2$ between any two of its vertices.
Apply \cref{clm:c1} to find a path $P_0$ with at least two vertices that is vertex-disjoint from $C_0$.  
Since $B$ is $2$-connected, by Menger's Theorem there exist two vertex-disjoint paths $R_1, R_2$ from $P_0$ to $C_0$, with no internal vertices in $P_0$ or $C_0$: say $R_i$ goes from $u_i \in V(P_0)$ to $v_i \in V(C)$.  By combining $R_1$, $R_2$, the path in $P_0$ between $u_1$ and $u_2$, and a path of length at least $2$ in $C_0$ between $v_1$ and $v_2$, we obtain a cycle $C_1$ of length at least $5$.

So we have a cycle $C_1$ of length at least $5$.  Applying \cref{clm:c1} again, there is a path $P_1$ with at least two vertices that is vertex-disjoint from $C_1$.  There exists a path $P_2$ from $P_1$ to $C_1$, with no internal vertices in $P_1$ or $C_1$: say $P_2$ goes from $w \in V(P_1)$ to $x \in V(C_1)$.  Let $w'$ be a neighbor of $w$ in $P_1$.  Then $\{w'w\} \cup P_2 \cup C_1$ contains a $Y$ subgraph, a contradiction.  Therefore $B$ cannot have unbounded path length.

The conclusion for finite unions of blocks follows from \cref{obs:pathblock}.
\end{proof}

We will show that all blocks are beads, and the following lemma will be used to show that they must be connected together at primary vertices.  A \emph{bead graph} is a graph that is the underlying graph of a bead, i.e., a bead without designated primary vertices.

\begin{lemma}\label{lem:addtobead}
Suppose $B$ is a (possibly infinite) bead graph.
Suppose we take $A \subseteq V(B)$ and add a path of length $2$ attached to one vertex of $A$, and a pendant edge at each remaining vertex of $A$.  
If the resulting graph $H$ has no $Y$ subgraph, then there is $V_1$ with $A \subseteq V_1 \subseteq V(B)$ such that $(B, V_1)$ is a bead.
\end{lemma}

\begin{proof}
We show that if there is no $V_1$ with $A \subseteq V_1 \subseteq V(B)$ that can be considered as a set of primary vertices of $B$, then $H$ has a $Y$ subgraph.
If $B$ is $K_{1,1,0}$ then $V_1=V(B)$ is a set of primary vertices that contains any $A$.
If $B$ is $K_{1,1,1}$ (a triangle) then any two vertices can be a set $V_1$ of primary vertices, so we only need to consider the case where $A = V(B)$, and then $H$ contains a $Y$ subgraph.
For the remaining cases we use the following.

\setcounter{claim}{0}
\begin{claim}\label{clm:d1}
A $4$-cycle with a path of length $2$ attached to a vertex and a pendant edge attached to an adjacent vertex has a $Y$ subgraph.  This is easy to verify.
\end{claim}

If $B$ is $K_{1,1,2} \cong K_{2,1,1}$ or $K_{2,2}$ then $B$ has a unique $4$-cycle $C$, and either pair of opposite vertices of $C$ forms a set $V_1$ of primary vertices.  So the sets $A$ for which there is no $V_1$ must contain a pair of adjacent vertices of $C$, and then $H$ contains a $Y$ subgraph by \cref{clm:d1}.

If $B$ is $K_{1,1,t}$ or $K_{2,t}$ for $t \ge 3$ (including $t$ infinite), then the only possible set of primary vertices $V_1$ consists of the two vertices of degree greater than $2$.
If $A$ is not a subset of this $V_1$, then $H$ contains either a path of length $2$ attached to a vertex not in $V_1$, or a path of length $2$ attached to a vertex of $V_1$ and a pendant edge attached to a vertex not in $V_1$.  In either case, $H$ has a $Y$ subgraph by \cref{clm:d1}.  

If $B$ is $K_4$, let $a_1$ be the vertex of $A$ to which the path of length $2$ is attached.  If $|A|=1$ then $V_1=A$ satisfies the lemma.  If $|A| \ge 2$ then $B$ contains a $4$-cycle $C$ with $a_1$ and another vertex of $A$ adjacent in $C$, so $H$ has a $Y$ subgraph by \cref{clm:d1}.
\end{proof}

\begin{lemma}\label{lem:pendantpath}
If $G$ is connected and has unbounded path length, and $H$ is a subgraph of $G$ of bounded path length, then for any positive integer $k$ there is a path that begins at a vertex of $H$, is otherwise disjoint from $H$, and has length $k$.
\end{lemma}

\begin{proof}
Let $\ell$ be the maximum path length of $H$.
Since $G$ has unbounded path length, there is a path $Q$ of length at least $\ell + 2k$ in $G$.  If $Q$ intersects $H$, then by \cref{obs:pathblock} there is a subpath $R$ of $Q$ (before or after the subpath $Q \cap H$) of length at least $k$ that contains a vertex $v$ of $H$ but is otherwise disjoint from $H$.  If $Q$ does not intersect $H$, then there is a path $R_1$ from $v \in V(H)$ to $w \in V(Q)$, with no internal vertices in $H$ or $Q$; one of the two subpaths into which $w$ divides $Q$, call it $R_2$, must have length at least $k$, and so $R = R_1 \cup R_2$ is a path of length at least $k$ that contains $v$ but is otherwise disjoint from $H$.  In either case, the subpath of $R$ containing $v$ and of length $k$ is the desired path.
\end{proof}

Now we introduce some terms to describe graphs with unbounded path length.  A \emph{ray} or \emph{one-way-infinite path} is a path $v_0 v_1 v_2 \dots$.  A \emph{two-way-infinite path} is a path $\dots v_{-2} v_{-1} v_0 v_1 v_2 \dots$.  A \emph{one-way-infinite strand} is obtained by stringing together (possibly infinite) beads $B_0, B_1, B_2, \dots$ where $B_i$ is strung to $B_{i+1}$ for $i \ge 0$.  In a one-way-infinite strand the initial bead $B_0$ may be a $K_4$. A \emph{two-way-infinite strand} is obtained by stringing together (possibly infinite) beads $\dots, B_{-2},B_{-1},B_0, B_1, B_2, \dots$ where $B_i$ is strung to $B_{i+1}$ for $i \in \mZ$.  A two-way-infinite strand does not have $K_4$ as a bead.  We may also add (possibly infinitely many) optional spikes incident with primary vertices that are in exactly two beads, to obtain \emph{spiked one-way- and two-way-infinite strands}.

The family $\BinfU$ consists of all one-way- and two-way-infinite spiked strands.  All graphs in $\BinfU$ are connected and have unbounded path length.

\begin{proposition}\label{prop:unboundedpathl}
An infinite connected graph $G$ with unbounded path length has no subgraph isomorphic to $Y$ if and only if $G$ belongs to the family $\BinfU$.
\end{proposition}
\begin{proof}
The graphs in $\BinfU$ do not have $Y$ as a subgraph, by arguments similar to those used in the proof of \cref{thm:noY} in \cref{Sec:MainResult}.

Now suppose that $G$ is an infinite graph with unbounded path length that does not have $Y$ as a subgraph.
We first show that each block $B$ of $G$ is a bead graph.  By \cref{lem:boundedpathl}, $B$ has bounded path length.  Since $G$ has unbounded path length, by \cref{lem:pendantpath} there is a path $R$ that begins at a vertex $v$ of $B$, is otherwise disjoint from $B$, and has length $6$.
Now $B \cup R$ is a subgraph of $G$, and hence $Y$-free, and it has finite maximum path length at least $6$, and $B$ is one of its blocks.  Therefore, by \cref{cor:sn} it is a spiked strand or spiked necklace.  Since it has a path of length $6$ attached to $v$, it is not a spiked necklace, so it is a spiked strand.  Since its block-cutvertex tree is a path, it has no spikes, so it is a strand.
Thus, $B$ is a block of the strand $B \cup R$, and thus $B$ is a bead graph.

We now show that we can consider $B$ not just as a bead graph, but as a bead where $B$ is only attached to other blocks of $G$ at its primary vertices.  (Note that $B$ might actually be a pendant edge that will end up as a spike in our final structure.)  Let $A$ be the set of cutvertices of $G$ that lie in $B$, which includes $v$.  Then $v \in A$ and there is a subpath of $R$ of length $2$ attached to $v$.  Moreover, $G$ contains at least one edge not in $B$ incident with each vertex in $A-\{v\}$.  Together with $B$ this path and these edges form a subgraph $G'$ of $G$ which satisfies the conditions of \cref{lem:addtobead}, so we may consider $B$ as a bead, and $A$ is a subset of its primary vertices.

At this point we know that every block of $G$ is a bead and every cutvertex of $G$ is a primary vertex of each of its incident blocks.  However, we need to show that $G$ has a path-like structure.  So let $G'$ be the subgraph of $G$ obtained by deleting all pendant vertices and pendant edges of $G$.  Then $G'$ is still a connected graph of unbounded path length with no $Y$ subgraph.
We claim that in $G'$ each cutvertex belongs to at most two blocks.  So suppose that $v$ belongs to blocks $B_1, B_2, B_3$ in $G'$.  Let $i \in \{1,2,3\}$.  If $B_i$ is not just $K_{1,1,0}$, then it contains a path $v w_i x_i$ beginning at $v$.  If $B_i$ is $K_{1,1,0}$, then it is an edge $v w_i$, and since $B_i$ is a block of $G'$ it was not a pendant edge in $G$, so there is some edge $w_i x_i \in E(G)$, $x_i \ne v$.  Now we have three paths $v w_i x_i$ for $i \in \{1,2,3\}$ in $G$ (although not necessarily in $G'$), which form a $Y$ in $G$, which is a contradiction.

Let $T$ be the block-cutvertex tree of $G'$, with white vertices that are cutvertices of $G'$ (which are primary vertices) and black vertices representing blocks of $G'$ (which are beads).  We just showed that each white vertex has degree at most $2$, and each black vertex has degree at most $2$ because a bead has at most two primary vertices.  Thus, $T$ is a path, and because each block has bounded path length, $T$ must be infinite, otherwise $G'$ has bounded path length.  Therefore, $T$ is either a one-way- or two-way-infinite path.

If $T$ is a two-way-infinite path then $G'$ is a two-way-infinite strand, and the pendant edges that we deleted from $G$ to form $G'$ are incident with cutvertices, which must be primary vertices in $G'$, so $G$ is a spiked two-way-infinite strand.  If $T$ is a one-way-infinite path then $T=b_0 v_1 b_1 v_2 \dots$, where $b_i$, $i \ge 0$, represents a block $B_i$ and each $v_i$, $i \ge 1$, is a cutvertex of $G'$.   Now $G'$ is a one-way-infinite strand.  But we need to be careful about adding back the pendant edges to form $G$.  If $B_0$ is a bead with two primary vertices (i.e., not $K_4$), one primary vertex of $B_0$ must be $v_1$, so let $v_0$ be the other primary vertex of $B_0$ in $G$.  If we deleted pendant edges incident with $v_0$ in forming $G'$ from $G$, we treat one of those pendant edges as an additional bead $B_{-1}$, and let $G'' = B_{-1} \cup G'$; otherwise, let $G'' = G'$.  Now $G''$ is a one-way-infinite strand, and $G$ is obtained from $G''$ by adding pendant edges incident with cutvertices of $G''$, which are primary vertices belonging to exactly two beads of $G''$, so $G$ is a one-way-infinite spiked strand.  Thus, $G \in \BinfU$.
\end{proof}

Our results can now be summarized as follows.

\begin{theorem}\label{thm:infinite}
A possibly infinite connected graph $G$ has no subgraph isomorphic to $Y$ if and only if either
\begin{statement}
\item\label{infsmallG} $G$ is obtained from a graph with at most six vertices by adding (possibly infinitely many) optional clones of leaves, or 
\item\label{infinfam} $G$ belongs to the family $\BinfB \cup \BinfU$, i.e., $G$ is a spiked strand, spiked necklace, spiked one-way-infinite strand, or spiked two-way-infinite strand.
\end{statement}
\end{theorem}

We can now give infinite analogues for the theorems in \cref{sec:Subtree}.
If we forbid a path $P_k$, then the graphs of interest have bounded path length, so again our previous arguments apply, and the theorems are unchanged except that when we can have many vertices that are clones we can now have infinitely many.  We therefore do not restate \cref{obs:forbidP2,obs:forbidP3} or \cref{theorem:forbidP4,theorem:forbidP5}.

\cref{thm:forbidClaw}, where we forbid a claw $K_{1,3}$, also has an obvious infinite version, where we allow one-way- and two-way-infinite paths, so we do not restate that either.

Now, we give an infinite analogue of \cref{thm:forbidY1}. 

\begin{theorem}\label{thm:infforbidY1}
A possibly infinite connected graph does not have $Y_1$ as a subgraph if and only if it is one of the following: a graph of order at most $4$, a (possibly infinite) star, a path, a cycle, a one-way-infinite path, or a two-way-infinite path.
\end{theorem}

\begin{proof}
The listed graphs do not have $Y_1$ as a subgraph.

Suppose now that $G$ is a connected graph with no $Y_1$ subgraph.  If $G$ has bounded path length then the argument from the proof of \cref{thm:forbidY1} applies, and we have a graph of order at most $4$, a path, a cycle, or a (possibly infinite) star.  So we may assume that $G$ has unbounded path length.  If $G$ has a vertex $v$ of degree at least $3$, then we apply \cref{lem:pendantpath} to find a path of length $2$ attached to the subgraph consisting of $v$ and all its incident edges, which creates a $Y_1$ subgraph.  Hence, all vertices of $G$ have degree at most $2$, and $G$ is a one-way- or two-way-infinite path.
\end{proof}

Finally, we have an infinite analogue of \cref{thm:forbidY2}.
Let $\TSinfB$ be the extension of $\TS$ to possibly infinite graphs of bounded path length, where the star or stars attached to the end of a finite path may be infinite.
Let $\TSinfone$ be the family of graphs obtained by attaching a triangle or the center of a (possibly infinite) star to the initial vertex of a one-way-infinite path.

\begin{theorem}\label{thm:infforbidY2}
A possibly infinite connected graph does not have $Y_2$ as a subgraph if and only if it is one of the following: a (possibly infinite) star, a path, a cycle, a one-way- or two-way-infinite path, a member of $\TSinfB \cup \TSinfone$, a $K_{1,1,t}$ or $K_{2,t}$ bead ($t \ge 2$ and possibly infinite) with optional (possibly infinitely many) pendant edges attached at primary vertices, or a graph obtained from a graph of order at most $5$ by adding (possibly infinitely many) clones of leaves.
\end{theorem}

\begin{proof}
By similar arguments to those in the proof of \cref{thm:forbidY2}, the listed graphs do not have $Y_2$ as a subgraph.

Suppose now that $G$ is a connected graph with no $Y_2$ subgraph.  If $G$ has bounded path length then the arguments from the proof of \cref{thm:forbidY2} apply, and we have (a possibly infinite version of) one of the graphs in \cref{thm:forbidY2}.  So we may assume that $G$ has unbounded path length.

Assume $G$ has two vertices $v_1, v_2$ of degree at least $3$.  Let $H$ be a subgraph of $G$ consisting of a shortest $v_1 v_2$-path $P$, and two edges not on $P$ incident with each of $v_1$ and $v_2$.  We can apply \cref{lem:pendantpath} to find a path $Q$ in $G$ of length $5$ joined to $H$ only at one of the endpoints of $Q$, creating a subgraph $H'$ of $G$.  Now $H'$ is a finite graph with no $Y_2$ subgraph, at least two vertices of degree at least $3$, and a path of length $5$ intersecting the rest of $H'$ only at one of its endpoints.  But none of the graphs listed in \cref{thm:forbidY2} have these properties, so this is a contradiction.  Hence, $G$ has at most one vertex of degree at least $3$.

Suppose $G$ has exactly one vertex $v$ of degree at least $3$.  Each component $C$ of $G-v$ is either a finite path or a one-way-infinite path.  If $C$ is finite then it is adjacent to $v$ only at one or both of its endpoints.  If $C$ is infinite then it is adjacent to $v$ only at its endpoint.
Since $G$ has unbounded path length, at least one component $C_1$ is a one-way-infinite path.
Suppose there is a component $C_2$ joined to $v$ at both endpoints, so $C_2 \ne C_1$ is a finite path of length at least $1$.  If $C_2$ has length $2$ or more, or there is a component $C_3 \ne C_1, C_2$, then there is a $Y_2$ subgraph.  So $C_2$ is a path of length $1$, and $C_1$ and $C_2$ are the only components.  Thus, $G$ is a triangle attached to a one-way-infinite path.
Now suppose that all components are joined to $v$ only at one endpoint, so there are at least three components. If any component $C \ne C_1$ is not a single vertex, then we have a $Y_2$ subgraph.  So all $C \ne C_1$ are single vertices, and $G$ is a star attached to a one-way-infinite path.
In both situations $G \in \TSinfone$.

If $G$ has no vertices of degree at least $3$, then $G$ is a one-way- or two-way-infinite path.
\end{proof}

\begin{remark}\label{aoc}
Many results in infinite graph theory require the use of a choice principle: the Axiom of Choice, or a weak form of it such as K\"{o}nig's Lemma.
Our results in this section do not require the use of a choice principle.
In some situations we use special cases of results which in full generality do require a choice principle.
When we find an infinite path we do not use K\"{on}ig's Lemma, we just use the fact that we have a tree of maximum degree at most $2$.
When we use Menger's Theorem in \cref{lem:boundedpathl} we are using the version for infinite graphs with finite connectivity, which was proved by \erdos. just by considering a finite subgraph (see \cite[p.~337/391]{Kon90}).
In other cases we specifically chose arguments so as not to use a choice principle.
For example, we could have proved \cref{lem:domset} using an argument based on spanning trees and \cref{thm:treenoycat}, but constructing spanning trees in infinite graphs generally requires the Axiom of Choice.
We could have proved the unbounded path length case of \cref{thm:infforbidY2} by constructing an infinite edge-dominating path in a one-way- or two-way-infinite spiked strand (provided by \cref{prop:unboundedpathl}), but doing so would require choosing a path in each bead, so we used a different approach.
\end{remark}

\section{\turan. numbers, pathwidth, and growth constants}\label{sec:pathwidth}

In this section, we discuss the relevance of our results for larger questions regarding graphs with no tree as a minor or subgraph.  These include \turan. numbers and the \erdos.-\sos. conjecture, pathwidth, and growth constants.

The \erdos.-\sos. Conjecture \cite{Erd64} involves the \turan. number for trees.  The \emph{\turan. number $\ex(n,G)$} is the maximum number of edges in an $n$-vertex graph that does not contain a given graph $G$ as a subgraph. \erdos. and \sos. conjectured that for a $k$-vertex tree $T$, $\ex(n, T) \le (k-2)n/2$; graphs with all components $K_{k-1}$ show that this would be sharp for infinitely many $n$.  
A \emph{spider} is a tree obtained by subdividing edges of a star, so it consists of edge-disjoint paths joined at a central vertex; if there are $p$ paths of lengths $\ell_1, \ell_2, \dots, \ell_p$ we denote the spider by $S_{\ell_1, \ell_2, \dots, \ell_p}$.  \wozniak. \cite{Woz96} showed that the conjecture holds when $T$ is a spider of diameter at most $4$, i.e., when $\ell_i \le 2$ for all $i$.  Since $Y=S_{2,2,2}$, this means that the conjecture is true for $T = Y$ or any subtree of $Y$.  This also follows from our results.

More recently, Caro, \patkos., and Tuza \cite{CPT24} considered the \emph{connected \turan. number $\exc(n, G)$}, which is the maximum number of edges in a connected $n$-vertex graph that does not have $G$ as a subgraph.  They observe that for $2$-edge-connected graphs $G$, $\exc(n,G) = \ex(n,G)$, so considering $\exc(n,G)$ separately from $\ex(n,G)$ is only of interest for graphs that are disconnected or have a cutedge.  They determined formulas for $\exc(T,n)$ for a number of small trees $T$, including $Y$ and its subtrees.  Our results confirm the formulas for $Y$ and its subtrees.  In particular, $\exc(n,Y) = \binom{n}{2}$ for $n \le 6$, and $2n-2$ for $n \ge 7$.

Caro, \patkos., and Tuza point out that extremal graphs with no $Y$ subgraph are given by $K_n$ for $n \le 6$, and for $n \ge 7$ by a necklace that we will denote $A_n$, with two beads, one being $K_{2,1,1}$ and the other $K_{1,1,n-4}$.  They also claim \cite[Remark 3.6]{CPT24} that there are other extremal examples for $n \ge 7$.  Their examples are correct for $n=7$: there are exactly three $12$-edge graphs, namely $A_7$, a strand with two $K_4$ beads, and $K_5$ with two pendant edges at the same vertex.  However, their claim is incorrect for $n \ge 8$: \cref{thm:noY} implies that the only extremal example is $A_n$.  The other examples proposed in \cite{CPT24} actually contain $Y$ subgraphs.

Now we consider pathwidth of graphs with a forbidden tree minor.  A sequence $\mathcal{V}=(V_1,V_2,\dots, V_t)$ is a  \emph{path-decomposition} of a graph $G$ if $V_i\subseteq V(G)$ for $1 \le i \le t$; for every edge $e\in E(G)$ there is some $V_i$ that contains both endpoints of $e$; and for $1\le i<j<k\le t$,  $V_i\cap V_k\subseteq V_j$.
Each set $V_i$ is called a \emph{bag}.
The \emph{width} of $\mathcal{V}$ is $\max\{|V_i|-1 \;|\; 1\le i\le t\}$. 
The \emph{pathwidth} of $G$, denoted $\pw(G)$, is defined as the smallest width over all path-decompositions of $G$.

The notions of path-decompositions and pathwidth were defined by Robertson and Seymour in \cite{graphminorsI}, and they proved that the class of graphs obtained by forbidding a forest $F$ as a minor has bounded pathwidth. This was improved to a sharp bound as follows.

\begin{theorem}[Bienstock et al.~\cite{brst-quicklyexcforest}, short proof by Diestel \cite{d-shortproofpathwidththm}]\label{thm:pwnm2}
If $F$ is an $n$-vertex forest, then the pathwidth of $F$-minor-free graphs is at most $n-2$.  Furthermore, $K_{n-1}$ shows that this bound is sharp. 
\end{theorem}

Below we investigate the pathwidth of $Y$-minor-free graphs, but we first note that the graph $Y$ shows up in another context in the theory of pathwidth.  In \cite{kl-fmpw2}, Kinnersley and Langston observed that the graphs with pathwidth at most $1$ have $K_3$ or $Y$ as a forbidden minor, and are caterpillars (this is \cref{thm:treenoycat}).  Their main result was a characterization of graphs with pathwidth at most $2$ in terms of $110$ forbidden minors.

\cref{thm:pwnm2} shows that the sharp upper bound on pathwidth of $Y$-minor-free graphs is $7-2=5$.
However, there is a stronger upper bound for graphs in the family $\B$.  Since $K_4 \in \B$, this bound cannot be less than $3$, and in fact that is the correct bound.

\begin{proposition}\label{prop:pathwidthfam}
The pathwidth of a $Y$-minor-free graph is at most $5$, and the pathwidth of a graph in the family $\B$ is at most $3$.
\end{proposition}

\begin{proof}
As noted above, the pathwidth of $Y$-minor-free graphs is at most $5$, and there are graphs in $\B$ with pathwidth at least $3$.
We will now show that the pathwidth of a graph in the family $\B$ is at most 3.

For a strand, let $B_1, \dots, B_r$ be the beads in order along the strand, where $B_i$ has primary vertices $v_{i-1}$ and $v_i$ for $1 \le i \le r$.  (If $B_1 \cong K_4$ let $v_0=v_1$, and if $B_r \cong K_4$ let $v_{r}=v_{r-1}$.)
At each step we create bags containing all edges of $B_i$, then bags containing all spikes incident with $v_{i}$, then bags containing all edges of $B_{i+1}$, and so on.  All bags containing edges of $B_i$ contain both $v_{i-1}$ and $v_{i}$.
For a bead $B_1$ or $B_r$ isomorphic to $K_4$, we take a bag equal to $V(B_i)$.
For a bead $B_i$ isomorphic to $K_{2,1,1}$ we take two bags $\{v_{i-1}, w_1, w_2\}$ and $\{w_1, w_2, v_{i}\}$, where $w_1$ and $w_2$ are the secondary vertices.
For a bead $B_i$ isomorphic to $K_{1,1,0}$ we take a bag $\{v_{i-1}, v_{i}\}$.
For a bead $B_i$ isomorphic to $K_{1,1,t}$ or $K_{2,t}$ for $t \ge 1$, we label the secondary vertices as $w_1, w_2, \dots, w_t$ and take bags $\{v_{i-1}, w_1, v_{i}\}$, $\{v_{i-1}, w_2, v_{i}\}$, $\dots$, $\{v_{i-1}, w_t, v_{i}\}$.
If $v_i$ is incident with spikes $v_i x_1, v_i x_2, \dots, v_i x_s$ then we take bags $\{v_i, x_1\}$, $\{v_i, x_2\}$, $\dots$, $\{v_i, x_s\}$.
Processing beads and spikes in order along the strand gives the required path decomposition.

For a necklace, let $B_0, B_1, \dots, B_{r-1}$ be the beads in cyclic order, where $B_i$ again has primary vertices $v_{i-1}$ and $v_i$ (subscripts now interpreted modulo $r$).
We put $v_0$ into every bag of the decomposition, but otherwise process the beads and spikes in a similar way to a strand, starting with spikes incident with $v_0$, then edges of $B_1$, then spikes incident with $v_1$, and so on, finishing with spikes incident to $v_{r-1}$, and finally edges of $B_0$.  Since a necklace has no $K_4$ beads, all bags still have at most $4$ elements.

In both cases we have a path decomposition where all bags have at most $4$ elements, so the pathwidth of of the graph is at most $3$.
\end{proof}

In the proof of the previous lemma, bags of size $4$ are required only to deal with $K_4$ beads at the end of a strand, or to allow us to close up a necklace.  Thus, spiked strands with no $K_4$ beads have pathwidth at most $2$, and for every graph in $\B$ there is a set of at most two vertices (one secondary vertex in each $K_4$ bead for a spiked strand, or any primary vertex for a spiked necklace) whose deletion leaves a graph of pathwidth at most $2$.

The graphs in $\B$ have a smaller upper bound on pathwidth than the connected graphs obtained from graphs on at most $6$ vertices, which may have an arbitrarily large number of vertices, but have small radius and diameter.  We wonder if this is a general phenomenon.

\begin{question}
Suppose $T$ is a $k$-vertex tree that is not a path.  Do connected $T$-minor-free graphs with sufficiently high radius (or perhaps diameter) have a significantly stronger upper bound on pathwidth than $k-2$?  Can we obtain an even stronger bound by allowing deletion of a small number of vertices?
\end{question}

An answer to this question might help to resolve the following weakening of the \erdos.-\sos. Conjecture.

\begin{conjecture}\label{conj:esminor}
If $T$ is a $k$-vertex tree and $n \ge k \ge 2$, then the number of edges in an $n$-vertex $T$-minor-free graph is at most $(k-2)n/2$.
\end{conjecture}

There is also an obvious conjecture intermediate between the \erdos.-\sos. Conjecture and \cref{conj:esminor}.

\begin{conjecture}\label{conj:estopminor}
If $T$ is a $k$-vertex tree and $n \ge k \ge 2$, then the number of edges in an $n$-vertex graph that does not have $T$ as a topological minor (i.e., does not contain a subdivision of $T$ as a subgraph) is at most $(k-2)n/2$.
\end{conjecture}

\cref{conj:esminor,conj:estopminor} have the same sharpness examples as the \erdos.-\sos. Conjecture, namely graphs with all components isomorphic to $K_{k-1}$.  As far as we are aware, the best upper bound on the number of edges of an $n$-vertex graph without an arbitrary $k$-vertex tree $T$ as a subgraph (with $n \ge k \ge 2$) is $(k-2)n$, twice the conjectured bound.  If a graph has more edges than this, then it has average degree greater than $2(k-2)$, which means it has a subgraph with minimum degree greater than $k-2$ (repeatedly delete vertices of degree at most $k-2$), and hence at least $k-1$.  A copy of $T$ can then be found using a greedy argument.  However, \cref{thm:pwnm2} gives a slightly better upper bound for \cref{conj:esminor}.  It is well known that an $n$-vertex graph with treewidth (or pathwidth) at most $w$ and $n \ge w$ has at most $wn-\binom{w+1}{2}$ edges (the number of edges in a $w$-tree with $n$ vertices), which with $w=k-2$ gives a bound of $(k-2)n - \binom{k-1}{2}$ edges for $n \ge k$.

We note that Havet, Reed, Stein, and Wood \cite{HRSW20} have proposed a variant of the \erdos.-\sos. Conjecture, which also has natural weakenings for minors and topological minors.

In \cite{dhjmmw-excludedtree}, Dujmovi\'{c} et al.~give a result related to \cref{thm:pwnm2}, showing that every graph forbidding a tree $T$ as a minor can be represented as a subgraph of a relatively small structure.  Fix a tree $T$ with radius $r$, order $k$, and maximum degree $\Delta$.  Then there exist
constants $w$ and $c$ (depending on $T$) such that for every $T$-minor-free graph $G$ there is a graph $H$ of pathwidth at most $w$ such that $G$ is isomorphic to a subgraph of the strong product $H \boxtimes K_c$ (which is also the lexicographic product $H[K_c]$).  They show that $w = 2r-1$ and $c = (\Delta+r-2)(k-1)$ suffice in general, but for paths smaller values of $c$ and $w$ work, although we must have $w \ge r-1$.

\begin{question}
Can our characterizations be used to find small values of $c$ and $w$ that work for $T$-minor-free graphs where $T$ is $Y$ or one of its subtrees?  If so, does this provide insight into values of $c$ and $w$ that will work in more general situations?
\end{question}

Our characterization of graphs without $Y$ as a subgraph or minor is sufficiently simple that it would not be too difficult to find explicit generating functions for both labeled and unlabeled graphs in this class.  Here we just consider broad estimates of how quickly the numbers of these objects increase.
Using the notation of \cite[p.~243]{FS09}, we write $a_n \bowtie \alpha^n$ for a sequence $(a_n)$ if $\limsup_{n \to \infty} |a_n|^{1/n} = \alpha$.

We first consider growth constants for labeled graphs.  Bernardi, Noy, and Welsh \cite{BeNoWe10} say that a family of graphs with $g_n$ labeled elements of order $n$ has \emph{growth constant} $\gamma$ if $g_n/n! \bowtie \gamma^n$, i.e., $\limsup_{n \to \infty} (g_n/n!)^{1/n} = \gamma$.
The growth constant is known to be finite for proper minor-closed classes of graphs.  Bernardi, Noy, and Welsh observe that the growth constant of a minor-closed class of graphs $\G$ is the same as for the subset of $\G$ consisting of its connected elements.

If $T$ is a proper subtree of $Y$ then the growth constant of $T$-minor-free graphs is easily determined.  Bernardi, Noy, and Welsh point out that the growth constant of $P_k$-minor-free graphs is $0$ for any fixed $k$, so we only need to consider $T = K_{1,3}$, $Y_1$, or $Y_2$.  The $K_{1,3}$-minor-free graphs include 
paths, of which there are $n!/2$ on $n$ vertices, so the growth constant for $K_{1,3}$-minor-free graphs is at least $1$.  By \cref{thm:forbidY2} the connected $Y_2$-minor-free graphs with $n$ vertices are obtained from a graph of order at most $5$ by cloning leaves ($O(5^n)$ graphs), paths ($n!/2$ graphs), cycles ($(n-1)!/2$ graphs), stars ($n$ graphs), or spiked $K_{1,1,t}$ or spiked $K_{2,t}$ or $\TS$ graphs ($O(n^2 n!)$ graphs).  Thus, the total number of connected $n$-vertex $Y_2$-minor-free graphs is $O(n^2 n!)$ and the growth constant is at most $1$.  Since $K_{1,3} \subseteq Y_1 \subseteq Y_2$, the growth constant for all of $K_{1,3}$-, $Y_1$-, and $Y_2$-minor-free graphs is $1$.

For $Y$-minor-free graphs the results of Bernardi, Noy, and Welsh give a lower bound on the growth constant.  They consider \emph{thick caterpillars}, which are caterpillars where for each edge of the spine we can optionally add one vertex adjacent to both of its ends (forming a triangle), and show that their growth constant $\delta$ is approximately $2.251\,59$.  Thick caterpillars are spiked strands, so the growth constant of $Y$-minor-free graphs is at least $\delta$.  We determine the exact growth constant.  

\begin{theorem}\label{thm:labyfree}
Let $a_n$ be the number of $n$-vertex labeled $Y$-minor-free graphs.  Then $a_n/n! \bowtie \alpha^n$ where $\alpha$ is the reciprocal of the positive solution of $1-xe^x(2e^x-1-x+x^2/2) = 0$, so that $\alpha \approx 2.490\,805$.
\end{theorem}

\begin{proof}
As noted above, it suffices to consider only connected graphs.  The class of connected $Y$-free graphs is the union of three subclasses: $\A$ containing graphs covered by \crefWithTheorem{smallG}, $\Bs$ containing spiked strands, and $\Bn$ containing spiked necklaces.  There is some overlap between these classes, but this will not affect our analysis.  The number of labelings of elements of $\A$ is $O(n^6\, 6^n)$, which has growth constant $0$.  The growth constant of $\Bn$ is at most the growth constant of $\Bs$, because every $n$-vertex spiked necklace can be obtained from some $(n+1)$-vertex spiked strand by identifying two vertices.  So we need only consider the growth constant of $\Bs$.

It is not difficult to show that the exponential generating function for labelings of elements of $\Bs$ has the form $a(x) = a_1(x) + a_2(x)/(1-a_3(x))$ where $a_1(x), a_2(x), a_3(a)$ are entire functions (analytic at all complex values of $x$).  Here $a_1(x)$ deals with strands with a small number of beads, $a_2(x)$ deals with the ends of a strand, and $a_3(x)$ deals with adding one new primary vertex, spikes incident with that vertex, and secondary vertices of a new bead.  The growth constant $\alpha$ is the reciprocal of the radius of convergence of $a(x)$, and this radius is the smallest positive $x$ with $1-a_3(x)=0$ \cite[\S IV.3.2]{FS09}.  Thus, the exact expressions for $a_1(x)$ and $a_2(x)$ do not affect the growth constant, so we do not consider them further.  Now $a_3(x)$ is the product of three terms: $x$ for adding a new primary vertex, $e^x$ for adding spikes to that vertex, and a term representing secondary vertices of a bead.  This third term is the sum of $e^x$ for $K_{1,1,t}$ with $t \ge 0$, $e^x-1-x$ for $K_{2,t}$ with $t \ge 2$, and $x^2/2$ for $K_{2,1,1}$.  Thus, $a_3(x) = xe^x(2e^x-1-x+x^2/2)$.  The unique positive solution of $1-a_3(x)=0$ is $1/\alpha \approx 0.401\,476\,6$, so that $\alpha \approx 2.490\,805$.
\end{proof}

We can perform a similar analysis with unlabeled $Y$-minor-free graphs.  

\begin{theorem}
Let $b_n$ be the number of $n$-vertex unlabeled $Y$-minor-free graphs.  Then $b_n \bowtie \beta^n$ where $\beta$ is the larger positive solution of $x^4-3x^3+x^2-2x+1 = 0$, so that  $\beta \approx 2.852\,145$.
\end{theorem}

\begin{proof}
The proof is very similar to that of \cref{thm:labyfree}, except that we compute ordinary generating functions instead of exponential generating functions, so we omit most details.  
By \cite[proof of Theorem IV.8]{FS09}, it suffices to consider only connected graphs, because the general graphs are constructed from the connected ones by the \textsc{MSet} (multiset) operator of \cite{FS09}, which does not change the radius of convergence of the ordinary generating function.
The generating function that corresponds to $a_3(x)$ is $x\, (\frac 1{1-x}) ( \frac 1{1-x} + \frac{x^2}{1-x} + x^2)$.
\end{proof}

We obtain a larger constant for unlabeled graphs because $Y$-minor-free graphs in general have many automorphisms, and thus $a_n \ll n!\, b_n$.

\begin{question}
Can we find upper and lower bounds that are not too far apart for the growth constant for graphs without a fixed tree $T$ (or fixed forest $F$) as a minor? 
\end{question}

The following more general question may be relevant.

\begin{question}
Can we find upper and lower bounds that are not too far apart for the growth constant for graphs with treewidth or pathwidth at most $k$?
\end{question}

\bibliographystyle{hplain}
\bibliography{bibliography}{}
\vspace{-5mm}
\end{document}